\documentclass[11pt]{article}
\usepackage{PRIMEarxiv}

\usepackage{amsmath,amssymb,amsfonts,amsthm,bbm,mathrsfs,verbatim} 
\usepackage[misc]{ifsym}
\usepackage{graphics}                 
\usepackage{graphicx}
\usepackage{caption}
\usepackage{subcaption}
\usepackage{float}
\usepackage{array}
\usepackage{cancel}
\usepackage{booktabs}
\usepackage{color}                    
\usepackage{hyperref}                 
\usepackage{makeidx}
\usepackage{natbib}
\usepackage{geometry}  
\usepackage{setspace}
\usepackage{bm}
\usepackage{authblk}
\usepackage{hvfloat}
\usepackage{neuralnetwork}
\usepackage{algorithm}
\usepackage{algpseudocode}

\usepackage{multirow}
\usepackage{tikz}

\usepackage{fancyhdr}

\newcommand{\norm}[1]{\left\lVert#1\right\rVert}

\newtheorem{theorem}{Theorem}[section]
\newtheorem{proposition}[theorem]{Proposition}

\theoremstyle{definition}
\newtheorem{remark}[theorem]{Remark}
\newtheorem{definition}[theorem]{Definition}
\newtheorem{example}[theorem]{Example}

\def\N{{\mathbb N}}

\title{ Learning Contractive Integral Operators with Fredholm Integral Neural Operators}
\author[1]{Kyriakos C. Georgiou}
\author[2,\thanks{{Corresponding author: \texttt{constantinos.siettos@unina.it}}} 
]{Constantinos Siettos}
\author[3,\thanks{{Corresponding author: \texttt{ayannaco@aueb.gr}}} 
]{Athanasios N.\ Yannacopoulos}

\affil[1]{Department of Electrical Engineering and Information Technologies, \textit{University of Naples “Federico II”}, Naples, Italy }
\affil[3]{Dipartimento di Matematica e Applicazioni ``Renato Caccioppoli'', \textit{University of Naples “Federico II”}, Naples, Italy }
\affil[4]{Department of Statistics and Stochastic Modelling and Applications Laboratory, \textit{Athens University of Economics and Business}, Athens, Greece }

\graphicspath{{figures/}}

\begin{document}

\maketitle

\begin{abstract}
We generalize the framework of Fredholm Neural Networks, to learn  non-expansive integral operators arising in Fredholm Integral Equations (FIEs) of the second kind in arbitrary dimensions. We first present the proposed  Fredholm Integral Neural  Operators (FREDINOs), for FIEs and prove that they are universal approximators of linear and non-linear integral operators and corresponding solution operators. We furthermore prove that the learned operators are guaranteed to be contractive, thereby strictly satisfying the mathematical property required for the convergence of the fixed point scheme. Finally, we also demonstrate how FREDINOs can be used to learn the solution operator of non-linear elliptic PDEs, via a Boundary Integral Equation (BIE) formulation. We assess the proposed methodology numerically, via several benchmark problems: linear and non-linear FIEs in arbitrary dimensions, as well as a non-linear elliptic PDE in 2D. Built on tailored mathematical/numerical analysis theory, FREDINOs offer high-accuracy approximations and interpretable schemes, making them well suited for scientific machine learning/numerical analysis computations.
\end{abstract}


\section{Introduction}
In recent years, there has been rapidly growing interest in neural operators as a flexible and powerful framework for learning operators between function spaces, attracting considerable interest across scientific computing and applied mathematics. Following DeepONet \cite{lu2021learning}, which marked a new research era in operator learning with machine learning, a large body of work focused on learning evolution and solution operators for PDEs and dynamical systems, including Fourier NOs \cite{li2020fourier,kovachki2023neural}, Graph NOs \cite{li2020multipole}, Riemannian NOs \cite{peyvan2024riemannonets}, and, more recently, RandONets \cite{fabiani2025enabling}. The neural-operator literature has been driven primarily by methods that learn solution maps, or time-evolution operators associated with PDEs. By contrast, work that explicitly formulates the task as learning integral operators is much more limited, with notable works including Fourier NOs (FNOs) \cite{li2020fourier,kovachki2023neural}, Graph-NOs \cite{li2020multipole} and more recently, Neural Integral Equations \cite{zappala2024learning}. While DeepONets \cite{lu2021learning} are formulated as general operator-learning architectures, they can also be interpreted within the framework of integral-operator learning whenever the target solution map admits an integral representation. In such cases, the branch and trunk networks implicitly define a separable approximation of the kernel, leading to a finite-rank representation of the corresponding integral operator. In this context, Fourier NOs \cite{li2020fourier,kovachki2023neural} parametrize the integral kernel directly in Fourier space, while Graph NOs \cite{li2020multipole} learn kernel integral operators by combining a Nystrom
approximation with domain truncation, implemented with graph neural networks. In \cite{kovachki2023neural}, the authors provide a universal approximation theorem for these operators and show that they are discretization‑invariant, meaning the same model generalizes across different resolutions without retraining. 
In \cite{solodskikh2023integral}, the authors proposed ``Integral Neural Networks'', a class of deep networks that replace conventional discrete weight tensors with continuous weight functions and layer transformations defined via integral operations over hypercubes. More recently, in \cite{zappala2024learning}, the authors presented ``Neural Integral Equations'', which consist of replicating a fixed-point argument, learning the unknown, non-linear integral kernel using an iterative solver until convergence of the iterations. These approaches, are used as end-to-end, black-box operator approximators trained by nonlinear (typically gradient-based) optimization of network weights.

Here, building on our previous work on Fredholm NNs \cite{georgiou2025fredholm, georgiou2025fredholm2}, we advance both theoretically and computationally our numerical analysis-informed Fredholm NNs to the integral operator learning framework. Fredholm Neural Networks (FNNs/ Fredholm NNs) \cite{georgiou2025fredholm, georgiou2025fredholm2}, provide an interpretable numerical analysis-informed machine-learning framework for solving and learning linear and nonlinear integral equations through Krasnosel'skii--Mann fixed-point iterations. In this setting, the resulting scheme emulates a deep neural network whose depth and width, weights and biases are prescribed by the underlying mathematical theory \cite{georgiou2025fredholm}, while for the solution of the inverse problem the integral kernel is explicitly parameterized and trained end-to-end \cite{georgiou2025fredholm2}. This way, Fredholm NNs bridge the gap between fully black-box deep neural network-learning approaches and methods informed by numerical analysis.  It is worth recalling that the idea of constructing neural networks that emulate classical numerical analysis schemes can be traced back to \cite{rico1992discrete}, where the Runge--Kutta method was represented as a deep neural network and subsequently employed for the learning of discretized PDEs. This perspective has recently received renewed attention in the context of numerical analysis grounded machine learning, where neural architectures are designed so as to reflect the structure of established iterative and discretization methods. Building on the idea of Runge--Kutta neural networks, in \cite{doncevic2024recursively} proposed a recursively recurrent neural network (R2N2) superstructure that learns numerical fixed-point algorithms, including Krylov, Newton--Krylov, and Runge--Kutta schemes. Along similar numerical-analysis-informed lines, recent studies have also considered neural network architectures that can be systematically interpreted through existing numerical schemes, as in \cite{datar2025systematic} for linear dynamical systems. Within this line of research, in \cite{bacho2025operator}, the authors presented the Cholesky Newton-Kantorovich Neural Operator, a neural operator learning paradigm that emulates Newton–Kantorovich–type iterations by learning the Cholesky factors of the linearized PDE operator, enabling machine‑precision accuracy on both forward and inverse nonlinear PDE problems. Finally, in \cite{fabiani2025randonets}, we have proposed RandONets which may be viewed as an interpretable DeepONet variant, in which the branch and trunk networks are constructed from single-layer neural networks, while training is carried out through linear-algebra-based regularization techniques, such as Tikhonov regularization, and pivoted QR decomposition with regularization, thereby yielding a more structured and transparent operator-learning framework. More recently, RandONets have been used to learn operators of PDEs locally from spatio-temporal patches, and have been coupled with Newton-Krylov iterative methods to perform system level tasks \cite{fabiani2025enabling}.
 
Within this context, the proposed NO that we call Fredholm Integral Neural Operator (FREDINO) bridges custom/``white-box'' interpretable deep neural network architectures with a fixed point training scheme that learns and rigorously guarantees the resulting learned operator to be a contraction as required. In the FREDINO, the parameters, i.e., weights, biases, and activation functions, directly correspond to specific components of the integral equation. Consequently, learning the integral operator amounts to finding the optimal parameter values that define a Fredholm NN whose outputs align with the training data. Our scheme ensures that the learned operator is (locally) non-expansive. We assess this property also on unseen resolutions during the model testing phase. With the proposed approach, we are not only learning an operator that satisfies the specific integral equation for any input and output functions; we are essentially learning the entire method of successive approximations algorithm.
\par This paper is organized as follows: in section \ref{sec2} we provide the necessary theoretical background pertaining to integral operators, Fredholm Integral Equations and Fredholm Neural Networks. In section \ref{sec3} we prove that the proposed Fredholm Neural Operators are universal approximators for contractive integral operators and discuss how this setting ensures the learned integral operators are contractive. In section \ref{pde-sec}, we present an application of the FREDINO to non-linear elliptic PDEs via potential theory. Finally, in section \ref{sec5}, we assess FREDINOs via various linear and non-linear integral operators of Fredholm-type in aribitary dimensions as well as non-linear PDEs in 2D. We conclude and discuss further research directions in \ref{sec6}.

\section{Mathematical formulation and background}\label{sec2}
In this work, we will be considering both linear and nonlinear integral operators. We begin by considering linear integral operators of the form:
    \begin{equation}\label{OPERATOR-linear}
        (\mathcal{T}f)(x)  =   \int_{\mathcal{D}}K(x,z)f(z)dz,
    \end{equation}
    where $x,z \in  \mathcal{D}$, $\mathcal{D} \subset \mathbb{R}$ is a bounded and compact domain, $K:\mathcal{D} \times \mathcal{D} \rightarrow \mathbb{R}$ is the kernel function, $f: \mathcal{D} \rightarrow \mathbb{R}^d$, and the corresponding nonlinear operator:
    \begin{eqnarray}\label{OPERATOR-nlinear}
        (\mathcal{T}f)(x)  = \int_{\mathcal{D}}K(x,z)G(f(z))dz,
    \end{eqnarray}
where $G: \mathbb{R} \rightarrow\mathbb{R}$ is a Lipschitz function. Note that we will also use the notation $\mathcal{T}_K$ and $\mathcal{T}_{K,G}$, wherever it is important to highlight the specific choices. Such operators arise for example in Fredholm Integral Equations of the second kind, which are of the form:
\begin{eqnarray}\label{ie}
	f(x) = g(x) + \int_{\mathcal{D}}K(x,z) f(z)dz.
\end{eqnarray}
We then consider nonlinear integral operators of the form
\begin{eqnarray}\label{nl-ie-def}
    f(x) = g(x) + \int_{\mathcal{D}}K(x,z) G(f(z))dz, \,\,\, g: \mathcal{D} \rightarrow \mathbb{R}^d.
\end{eqnarray}

\par Throughout this paper, we will be considering the Hilbert space ${\mathcal H}:=L^{2}(\mathcal{D} ; 
\mathbb{R}^d)$ endowed with the inner product $\langle f, g \rangle =\int_{\mathcal{D}} f(x) g(x) dx$. (The results below are generalizable to other spaces, e.g., $\mathcal{C}(\mathcal{D}, \mathbb{R}^d)$ or Lebesgue spaces). We also consider the space  
\begin{equation}
\mathcal{Z} = L^{2}(\mathcal{D} \times \mathcal{D})
:=
\left\{
K : \mathcal{D} \times \mathcal{D} \to \mathbb{R}
\;\middle|\;
\int_{\mathcal{D}} \int_{\mathcal{D}} |K(x,z)|^{2} \, dx \, dz < \infty
\right\},
\end{equation}
with corresponding norm
\begin{equation}
    \|K\|_{L^{2}(\mathcal{D} \times \mathcal{D})}=\Big(\int_D \int_D |K(x,z)|^2 dx\, dz \Big)^{1/2},
\end{equation}
so that for the kernels we have $K \in \mathcal{Z}= L^{2}(D \times D)$. Therefore, we define the set:
\begin{equation}
{\mathbb T}:=\{{\cal T}_{K} : {\cal H} \to {\cal H} \,\, : \,\, \exists K \in {\mathcal Z} \quad \mbox{such that} \quad {\cal T}_{K}u(x)=\int_{{\cal D}} K(x,z) u(z) dz\},
\end{equation}
consisting of compact operators. Extensions are of course possible.
Importantly, using the Cauchy-Schwarz inequality, we have :
\begin{equation}
    | \big(\mathcal{T} u\big)(x) |\leq \|K\|_{L^2(\cal D)}\, \|u\|_{L^2(\cal D)}.
\end{equation}
Squaring and integrating the above w.r.t to $x$, we get:
\begin{equation}
   \int_D | \big(\mathcal{T}_{K} u\big)(x) |^2\, dx \leq \Big(\int_D\|K(x,\cdot)\|^2_{L^2(\cal D)} dx \Big) \|u\|^2_{L^2(\cal D)} = \|K \|^2_{L^2(\cal D \times \cal D)} \|u\|^2_{L^2(\cal D)}.
\end{equation}
Therefore, we have:
\begin{equation}
 \| \mathcal{T}_K u\|^2_{L^2(\cal D)}  \leq \|K\|^2_{L^2(\cal D \times \cal D)} \, \|u\|^2_{L^2(\cal D)}.
 \label{eq:supremuminequality}
\end{equation}
Now taking the \textit{supremum} over all $u\neq 0$, we get:
\begin{equation}
    \|\mathcal{T}_K\|_{L^2 \to L^2}=\sup_{u \neq 0}\frac{\|\mathcal{T}_K\, u\|_{L^2(\cal D)}}{\|u\|_{L^{2}(\cal D)}}\le \|K\|_{L^2(\cal D \times \cal D)}.
\end{equation}
%

\begin{remark}
Since $\mathcal{D}$ is bounded and compact, in the case $K$ is continuous, we will also be considering the Banach space $\mathcal{X}:= \mathcal{C}\big(\mathcal{D}; \mathbb{R}^d \big)$, endowed with the norm $\|f(x)\|_{\infty} = \sup_{x \in \mathcal{D}} |f(x)|$, with $x \in \mathbb{R}^d$. Then, the main results that are shown below will also hold in for $\mathcal{T}_K : \mathcal{X}\rightarrow \mathcal{X}$.
\end{remark}
Hereinafter, we will use $\| \cdot \|$ to represent the norms in either of the aforementioned spaces, using the subscript when necessary to distinguish for clarity. 
\color{black}

In this work, we tackle the problem of learning the integral operator by modelling the unknown kernel function $K$ and non-linearity $G$ (using neural networks) from information of the solution $f$ of the equations \eqref{ie} or \eqref{nl-ie-def}, to provide a consistent finite‐dimensional approximation of the integral operator. We will see how we can do this using a forward pass through Fredholm Neural Networks. Therefore, we provide background on non-expansive integral operators, which are used in Fredholm NNs, below.

\subsection{Non-expansive integral operators and FIEs}
We start by recalling the definition of a contraction operator.
\begin{definition}[Contractive operator]
    Consider an operator $\mathcal{T}: \mathcal{X} \rightarrow \mathcal{X}$ (equivalently $\mathcal{T}:\mathcal{H} \rightarrow \mathcal{H}$). We say that $\mathcal{T}$ is a $q-$contraction if:
    \begin{eqnarray}
        \|\mathcal{T}f_1 - \mathcal{T}f_2\| \leq q \|f_1 - f_2\|,
    \end{eqnarray}
    for some $q < 1$. If $q=1$, then we say the operator is non-expansive.
\end{definition}

\par When $\mathcal{T}$ is a contraction, it is well known that the FIE can be solved using the Neumann series $f(x) = \sum_{n=0}^{\infty} (\mathcal{T}^{(n)}g)(x)$, whose convergence is guaranteed by the Banach Fixed-Point Theorem. Simultaneously, such equations can also be solved numerically utilizing appropriate discretization schemes, such as the Nyström method. Furthermore, we recall the Krasnoselskii-Mann (KM) method (see e.g., \cite{kravvaritis2020variational} for further details). In the particular case, when $q=1$, we consider the Hilbert space ${\cal H}$. 
\begin{proposition}[Krasnosel'skii-Mann (KM) method]\label{km-method}  Consider the 
FIEs \eqref{ie} or \eqref{nl-ie-def} defined by a non-expansive operator $\mathcal{T}:{\cal H}\rightarrow \mathcal{H}$ and a sequence $\{\kappa_n\}, \kappa_n \in (0,1)$,  such that $\sum_n \kappa_n(1-\kappa_n) = \infty$. Then, the iterative scheme:
\begin{eqnarray}\label{km-it}
    f^{(n+1)}(x) = (1-\kappa_n)f^{(n)}(x) + \kappa_n (g(x) + \mathcal{T} f^{(n)})(x) = f^{(n)}(x) + \kappa_n \big(g(x) + (\mathcal{T}f^{(n)})(x) -f^{(n)}(x)\big) 
\end{eqnarray}
with $f_0(x) = g(x)$, converges to the fixed point solution of the FIE, $f^{*}(x)$.
\end{proposition}
\par It is straightforward to see that, when $\mathcal{T}$ is a contraction, we can set $\kappa_n =1$, for all $n$, in which we can work on the original Banach space $\cal{X}$ and obtain:
\begin{eqnarray}\label{iterations}
	f^{(n)}(x)= g(x) + \int_{\mathcal{D}} K(x,z) f^{(n-1)}(z)dz, \,\,\ n \geq 1,
\end{eqnarray}
which converges to the fixed point solution. This is the Picard iteration (method of successive approximations), which is easily seen to coincide with the Neumann series expansion. 
\par The Fredholm NN framework presents a novel way of presenting DNNs, by establishing a connection between the model architecture and the KM and Picard iterations for solving FIEs. This theory has been developed in \cite{georgiou2025fredholm, georgiou2025fredholm2}. For completeness, we now present the main results in detail below.

\subsection{Feedforward Fredholm NNs}
As a first step, it is useful to define the following discretized version of the fixed point approximation, which is the basis for the construction of the Fredholm NN \cite{georgiou2025fredholm}. We note that in \cite{georgiou2025fredholm} the approach was introduced only for 1D FIEs; as mentioned above, this is extended in the current work as we will also be considering higher dimensional problems. Therefore the definitions below are generalized to account for this extension. 

\begin{definition}[Discretized Krasnoselskii-Mann operator]
Consider a function $f$ satisfying the FIE \eqref{ie} or \eqref{nl-ie-def}. Consider a $z-$grid $Z = \{z_1, \dots z_N \}$, with $z_i \in \mathbb{R}^d$, $|Z| = N$ and such that $x \in Z$. Then, we define the following discrete operator $(\mathcal{P}_N$ representing a single iteration of the KM-algorithm using the discretized version of the integral operator 
$\mathcal{T}$
(with fixed $g :\mathbb{R} \rightarrow \mathbb{R}$), by:
\begin{flalign}\label{disc-op}
	(\mathcal{P}_N f) (x) := (1-\kappa_n)f(x) &+ \kappa_n \big(g(x) + \sum_{i =1}^N K(x,z_j)f(z_j) \Delta z \big) \notag \\ &= \kappa_n g(x) + \sum_{j =1}^N f(z_j)\Big( K(x,z_{j})\kappa_n \Delta z + (1-\kappa_n) \mathbbm{1}_{\{z_j = x\}}\Big).
\end{flalign}
\end{definition}

\begin{remark}
   Note that in the definition above the grid $Z$ can be constructed in various ways, e.g., uniformly or using quasi-Monte Carlo sequences. Indeed, when considering FIEs in higher dimensions we will use a low discrepancy sequence, namely Sobol sampling for the numerical approximation to the integral. 
\end{remark}

\begin{definition}[$M-$layer fixed point estimate] \label{dd}
We define the $M-$layer fixed point estimate of the solution to the FIE as the following representation arising from applying the discretized KM operator $\mathcal{P}_N$ for the given choice of $g(x)$, $M$ times:
\begin{align} \label{m-operator}
f^{(M)}(x) = (\mathcal{P}_N^{(M)} g)(x) = \mathcal{P}_{N} \circ (\mathcal{P}_{N} \circ (\dots \circ (\mathcal{P}_{N}g)(x))) \end{align}
\end{definition}

\par The above definitions give rise to the connection between fixed point approximations and neural networks that is the basis of Fredholm NNs. Notice that the form of the discretized KM operator in \eqref{disc-op} matches the linear transformations that connect the output of one layer of a fully connected neural network to the next layer. Finally, composing the operator as in \eqref{m-operator} becomes equivalent to ``passing through'' the neural network to obtain the output, which is the fixed point approximation to the solution of the FIE.

\begin{proposition}
[Fredholm Neural Network construction]\label{DNN_construction}
Consider a KM constant $\kappa = \kappa_n$ for all $n$, and a discretization $Z = \{ z_1, \dots, z_N\}$ of the domain $\mathcal{D} \subset \mathbb{R}^d$. 
The $M$-layer approximation of the solution to the linear FIE \eqref{ie} can be implemented as a DNN with input $x \in \mathbb{R}^d$, $M$ hidden layers, a linear activation function and a single output node, where the weights and biases are given by:
\begin{flalign}
    W_1 = \big(\begin{array}{ccc}
		\kappa g(z_1), \dots, \kappa g(z_{N})
	\end{array}\big)^{\top}, \,\,\,\,\    b_1 = \left(\begin{array}{ccc}
		0, 0, \dots, 0
	\end{array}\right)^{\top}
 \end{flalign} 
for the first hidden layer,  
\begin{eqnarray}	\label{inner-weight}
W_m=
\left(\begin{array}{cccc}
	K_D\left(z_1\right) & {K}\left(z_1, z_2\right)\kappa \Delta z & \cdots & {K}\left(z_1, z_{N}\right)\kappa \Delta z \\
 {K}\left(z_2, z_1\right)\kappa \Delta z  & K_D\left(z_2\right) & \cdots & {K}\left(z_2, z_{N}\right)\kappa \Delta z \\
	\vdots & \vdots & \ddots & \vdots \\
	\vdots & \vdots & \vdots & \vdots \\
	{K}\left(z_{N}, z_1\right)\kappa \Delta z & {K}\left(z_{N}, z_2\right)\kappa \Delta z & \cdots & K_D\left(z_{N}\right) 
\end{array}\right),
\end{eqnarray}
and
\begin{eqnarray}
	b_m=\left(\begin{array}{ccc}
		\kappa g(z_1), \dots, \kappa g(z_{N})
	\end{array}\right)^{\top},
\end{eqnarray}
for hidden layers $m= 2, \dots, M-1$, where $K_D\left(z\right) := {K}\left(z, z\right)\kappa \Delta z + (1-\kappa)$, and:
\begin{flalign} \label{outer-weight}
	\begin{gathered}
		W_M=\big(\begin{array}{ccc}
			K(z_1, x)\kappa \Delta z, \dots, K(z_{i-1},x)\kappa \Delta z, K_D(x), K(z_{i+1}, x)\kappa \Delta z, \dots, K(z_{N}, x)\kappa \Delta z
		\end{array}\big)^{\top},
	\end{gathered}
\end{flalign}
$b_M =\big(\kappa g(x) \big)$, for the final layer, assuming $z_i = x$. 
\par Then, letting $\mathcal{A}_m$ be the map defined by $\mathcal{A}_m(x) = W_m x + b_m$, for $m=1,\dots, M$, the output of the neural network is given by:
\begin{equation}
    \hat{f}(x)
    = \Big(
        \mathcal{A}_M \circ \cdots \circ \mathcal{A}_1
      \Big)(x).
\end{equation}
\end{proposition}

Based on the constructions above we can formulate the Universal Approximation Theorem for Fredholm Neural Networks (Fredholm NNs) below (see \cite{georgiou2025fredholm} for the proof for the one dimensional case, and we note that the $d-$dimensional result follows easily by following the same steps).

\begin{theorem}[Fredholm NNs as universal approximators of linear FIE solutions \cite{georgiou2025fredholm}]\label{FUA}
Consider the linear integral operator, $\mathcal{T} : \mathcal{V} \rightarrow \mathcal{V}$ defined in \eqref{OPERATOR-nlinear}, where ${\cal V}$ is either ${\cal X}$ or ${\cal H}$.
Consider the space $\mathcal{V'} \subset \mathcal{V}$ defined by $\mathcal{Z'}: = \{ f \in \mathcal{V}: g+ \mathcal{T}f = f, \text{ for some } g, K \}$. Then, for any $f \in {\mathcal Z}'$ and for any given approximation error $\epsilon$, there exists a fully connected Fredholm NN (denoted by $\mathcal{F}(x;g, K)$) with $M$ hidden layers as described in Proposition \ref{DNN_construction}, and activation function given by:
\begin{eqnarray}
 \sigma(x) =
\begin{cases}
   \kappa g(x), \text{ for layer }  = 1, \\
    x, \text { for layer} \geq 2,   
\end{cases}
\end{eqnarray}
 such that $\|f(x) - \mathcal{F}(x;g, K)\|_{L^2(\mathcal{D})} \leq \epsilon$.
\end{theorem}

Note that the proof of Theorem \ref{FUA} as given in \cite{georgiou2025fredholm} generalizes trivially for $x\in \mathbb{R}^d$, as required.

\subsection{Recurrent Fredholm Neural Networks}

 We now turn to the non-linear FIE of the form \eqref{nl-ie-def}. As briefly discussed in \cite{georgiou2025fredholm}, the approximation of the fixed point solution can again be modeled with a Fredholm NN, but requires an additional recurrent layer for the application of the function $G(\cdot)$. This is a recurrent neural network (RNN), and in particular an Elman network structure \cite{elman1991distributed}. Although in \cite{georgiou2025fredholm}, we considered an iterative process for the non-linear FIE case, here we focus on this RNN approach. This structure is shown in Fig. \ref{fig:DNN-grid}. 
 
 \par An important note is that we do not use the KM algorithm in this case, but rather the Picard iteration, i.e,:
 \begin{equation}\label{picard-nl}
     (\mathcal{P}_N f) (x) := g(x) + \sum_{j =1}^N  K(x,z_j)f(z_j) \Delta z.
 \end{equation} 
 The reason is that the implementation of the KM algorithm as a Fredholm NN requires the rearangement as shown in \eqref{disc-op}. However, applying the non-linearity $G(\cdot)$ to the summand means that this no longer holds, and therefore the KM algorithm cannot be represented as a Fredholm Neural Network connections as described in the linear case. The difference in construction is described in detail in the result below. 

 \begin{proposition}
[Recurrent Fredholm NN construction]\label{DNN_construction-rec} The $M$-layer approximation of the solution of the non-linear FIE \eqref{nl-ie-def} can be implemented as a RNN with input $x\in \mathbb{R}^d$, $M$ hidden layers, a linear activation function and a single output node, where the weights and biases are given by:
\begin{flalign}
    W_1 = \left(\begin{array}{ccc}
		 g(z_1), \dots,  g(z_{N})
	\end{array}\right)^{\top}, \,\,\,\,\    b_1 = \left(\begin{array}{ccc}
		0, 0, \dots, 0
	\end{array}\right)^{\top}
 \end{flalign} 
for the first hidden layer,  
\begin{eqnarray}	\label{inner-weight}
W_m=
\left(\begin{array}{cccc}
	K\left(z_1, z_1\right) & {K}\left(z_1, z_2\right) \Delta z & \cdots & {K}\left(z_1, z_{N}\right) \Delta z \\
 {K}\left(z_2, z_1\right) \Delta z  & K\left(z_2, z_2\right) & \cdots & {K}\left(z_2, z_{N}\right) \Delta z \\
	\vdots & \vdots & \ddots & \vdots \\
	\vdots & \vdots & \vdots & \vdots \\
	{K}\left(z_{N}, z_1\right) \Delta z & {K}\left(z_{N}, z_2\right) \Delta z & \cdots & K\left(z_{N}, z_N\right) 
\end{array}\right),
\end{eqnarray}
and
\begin{eqnarray}
	b_m=\left(\begin{array}{ccc}
		 g(z_1), \dots, g(z_{N})
	\end{array}\right)^{\top},
\end{eqnarray}
for hidden layers $m= 2, \dots, M-1$, and:
\begin{flalign} \label{outer-weight}
	\begin{gathered}
		W_M=\left(\begin{array}{ccc}
			K(z_1, x)\Delta z, \dots, K(z_{N}, x) \Delta z
		\end{array}\right)^{\top},
	\end{gathered}
\end{flalign}
$b_M =\big( g(x) \big)$, for the final layer.  
\par Then, considering the recurrent input non-linearity $G(\cdot):\mathbb{R}\rightarrow\mathbb{R}$ applied to each node and letting $\mathcal{A}_m$ be the map defined by $\mathcal{A}_m(x) = W_m G(x) + b_m$, for $m=1,\dots, M$, the output of the neural network is given by:
\begin{equation}
    \hat{f}(x)
    = \Big(
        \mathcal{A}_M \circ \cdots \circ \mathcal{A}_1
      \Big)(x).
\end{equation}
\end{proposition}




The corresponding Universal Approximation Theorem can be formulated as follows.

\begin{theorem}[Fredholm NNs as universal approximators of non-linear FIE solutions]\label{FUA}
Consider the contractive integral operator ${\cal T}: {\cal X} \to {\cal X}$, defined in \eqref{OPERATOR-nlinear}.
Consider the space $\mathcal{X'} \subset \mathcal{X}$ defined by $\mathcal{X'}: = \{ f \in \mathcal{X}: g+ \mathcal{T}f = f, \text{ for some } g, K, G \}$. Then, for any $f \in {\cal X}'$ and for any given approximation error $\epsilon$, there exists a fully connected DNN (denoted $\mathcal{F}(x;g, K, G)$) with $M$ hidden layers as described in Proposition \ref{DNN_construction-rec}, with recurrent input non-linearity function $G$, activation function given by:
\begin{eqnarray}
 \sigma(x) =
\begin{cases}
   \kappa g(x), \text{ for layer }  = 1, \\
    x, \text { for layer} \geq 2,   
\end{cases}
\end{eqnarray}
such that $\|f(x) - \mathcal{F}(x;g, K, G)\|_{L^2(\mathcal{D})} \leq \epsilon$.
\end{theorem}

\begin{figure}
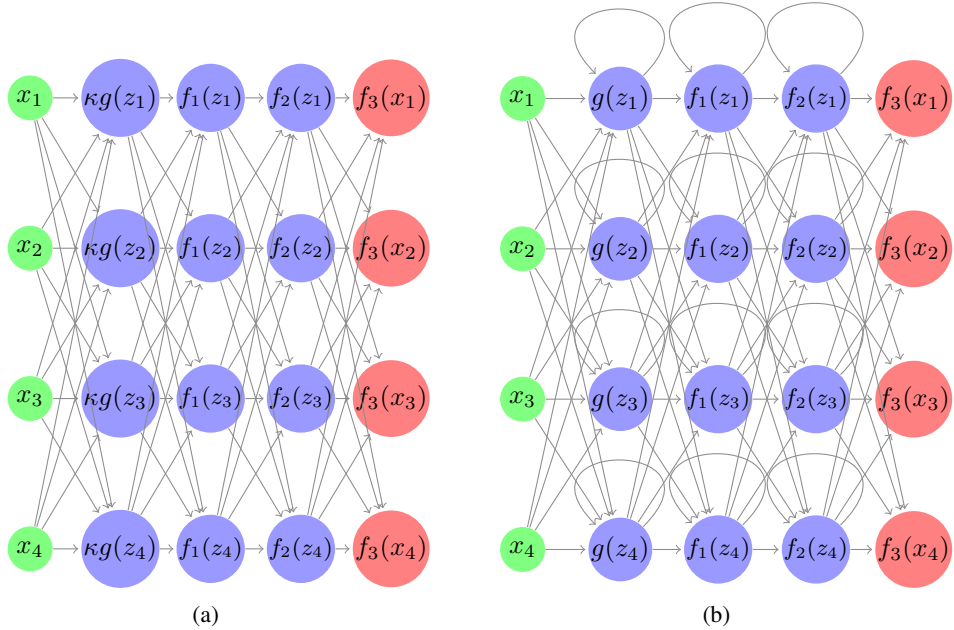

\centering
\subfloat[]{ \begin{neuralnetwork}[height=1.5, layertitleheight=2.5cm, nodespacing=2.0cm, layerspacing=1.2cm]
        \newcommand{\x}[2]{$\kappa g(z_#2)$}
        \newcommand{\z}[2]{$x_#2$}
        \newcommand{\y}[2]{$f_3(x_#2)$}
        \newcommand{\hfirst}[2]{\small $f_1(z_#2)$}
        \newcommand{\hsecond}[2]{\small $f_2(z_#2)$}
        \inputlayer[count=4, bias = false, text = \z]
        \hiddenlayer[count=4, bias=false,  text=\x] \linklayers
        \hiddenlayer[count=4, bias=false,  text=\hfirst] \linklayers
        \hiddenlayer[count=4, bias=false,  text=\hsecond] \linklayers
        \outputlayer[count=4,  text=\y] \linklayers
    \end{neuralnetwork}} \,\,\,\,\,\,\,\,\,\,
    \subfloat[]{%
\begin{neuralnetwork}[height=1.5, layertitleheight=2.5cm,
                      nodespacing=2.0cm, layerspacing=1.3cm]
  \newcommand{\x}[2]{$ g(z_#2)$}
  \newcommand{\z}[2]{$x_#2$}
  \newcommand{\y}[2]{$f_3(x_#2)$}
  \newcommand{\hfirst}[2]{\small $f_1(z_#2)$}
  \newcommand{\hsecond}[2]{\small $f_2(z_#2)$}
  \inputlayer[count=4, bias=false, text=\z]
  \hiddenlayer[count=4, bias=false, text=\x]        \linklayers   
  \hiddenlayer[count=4, bias=false, text=\hfirst]   \linklayers   
  \hiddenlayer[count=4, bias=false, text=\hsecond]  \linklayers   
  \outputlayer[count=4, text=\y]                    \linklayers   

  \foreach \L in {1,2,3}{%
    \foreach \N in {1,...,4}{%
      \link[from layer=\L,from node=\N,%
            to layer=\L,to node=\N,%
            style={->,in=135,out=45,looseness=7}]%
    }%
  }%
\end{neuralnetwork}}

    \caption{Schematic of: (Left) The Feedfroward Neural Network (Right) The Recurrent Fredholm Neural Network. The recurrent layer applies the non-linearity $G(\cdot)$ pointwise to the output of each node. The grid points $z_i \in \mathbb{R}^d$ can be sampled either uniformly or using quasi-Monte Carlo methods such as Sobol sampling. }
\label{fig:DNN-grid}
\end{figure}

\section{FREDINOs: Fredholm Integral Neural Operators as universal approximators}\label{sec3}

In this section, we provide the main theoretical results for the proposed method. Recall that we consider the problem of learning the unknown, non-expansive, integral operators \eqref{OPERATOR-linear} or \eqref{OPERATOR-nlinear}. By using the Fredholm NN framework, we can tackle this problem by learning the unknown functions $K:\mathcal D \times D \rightarrow \mathbb{R}$ and $G:\mathbb{R} \rightarrow \mathbb{R}$, such that: (a) the integral equation is satisfied, (b) the learned integral operator, defined by the approximated kernel and non-linearity, is strictly non-expansive.

\par To accurately set-up the universal approximation theorems that motivate Fredholm Neural Operators (Frehdolm NOs), we first discuss the nature of the integral equations. Specifically, we highlight that there exist infinitely many kernels $K$ (and functions $G$ in the non-linear case) that can be used to define the same integral operator, in this setting. Indeed, considering first the linear case, let $\{g_i, f_i\}$ for $i = 1, \dots, N$ be the (finite) pairs that satisfy the IE defined by the integral operator with kernel $K$, i.e., 
\begin{equation}
    f_i(x) = g_i(x) + (\mathcal{T}_K f_i)(x),
\end{equation}
for all $i = 1,\dots, N$, where $\mathcal{T}_Kf_i(x) :=  \int_{\mathcal{D}}K(x,y)f_i(z)dz$. Then, if we define the map $\mathcal{A}: L^2(\mathcal{D} \times \mathcal{D}) \rightarrow L^2(\mathcal{D})$ with $\mathcal{A}( \Phi )= \left(\mathcal{T}_{\Phi}f_1, \dots, \mathcal{T}_{\Phi}f_N \right)$, the null space of the map reads: $\ker(\mathcal{A}) = \{\Phi: \mathcal{T}_{\Phi}f_i = 0, \text{for all }i = 1,\dots, N\}$. Trivially, $\ker(\mathcal{A})$ is non-empty, and hence, for any $\Phi \in \ker(\mathcal{A})$, we can define $\tilde{K}:= K + \Phi$ and we have:
\begin{equation}
 f_i(x) = g_i(x) + (\mathcal{T}_{\tilde{K}} f_i)(x),
\end{equation}
for all $i = 1, \dots, N$. 
\par Non-uniqueness holds also in the non-linear case, where we want to learn both $K$ (one of the infinite kernels that satisfy the integral equation) and $G$ to define the integral operator. One can follow a similar argument as above. 
\par These observations motivate the formulation of the universal approximation theorems that follow in terms of the family of kernels $\mathcal{K}$ (and non-linearities $\mathcal{G}$), satisfying the IEs. We begin with the linear case.

\begin{theorem}[FREDINOs as universal approximators of linear integral operators]\label{theorem} Consider the Hilbert space ${\mathcal H}=L^{2}(\mathcal{D} ; \mathbb{R}^d)$ and let the integral operator ${\cal T}_{K} : {\cal H} \to {\cal H},$ be defined by $\mathcal{T}_{K}u = \int_{\mathcal{D}}K(x,z)u(z)dz$, for any $u \in \mathcal{H}$, with kernel $K:\mathcal{D}\times \mathcal{D} \rightarrow \mathbb{R}$. Furthermore, let $\mathcal{S}_K$ represent the solution operator of the corresponding FIE $f(x) = g(x) + (\mathcal{T}_{K}f)(x)$, for any $g \in {\cal H}$, i.e., $(\mathcal{S}_{K})g = f$. Then: 
\begin{itemize}
    \item[1.] There exists a surrogate model (a neural network) for the kernel, $K_{ \theta}: \mathcal{D}\times \mathcal{D} \rightarrow \mathbb{R}$, for some parameters $\theta \in \Theta$, where $\Theta \subset {\mathbb R}^{P}, P \in \mathbb{N}$ such that: 
    \begin{equation}
        \| \mathcal{T} - \mathcal{T}_{K_{\theta}} \| < \epsilon_1,
    \end{equation}
    for any $\epsilon_1>0$, where $\mathcal{T}_{K_{\theta}}$ is the learned operator, defined by:
    \begin{equation}
    (\mathcal{T}_{K_{\theta}}u)(x) := \int_{\mathcal{D}} K_{\theta}(x,z) u(z) dz.
\end{equation}
Similarly, for the learned solution operator $S_{K_{\theta}}g$, we have for any $\epsilon_2>0$:
\begin{equation}
   \|S_K\,g - S_{K_{\theta}} \, g\|_{L^2(\mathcal{D})}   < \epsilon_2.
   \label{eq:SolOperCOnv}
\end{equation}
    \item[3.] If $ \mathcal{T}$ is a contraction, with constant $\rho <1$, then the learned operator $\mathcal{T}_{\theta}$ is also contractive, i.e., for any $u,v\in \cal H$:

\begin{equation}
    \| \mathcal{T}_{K_{\theta}} u  - \mathcal{T}_{K_{\theta}} v \|_{L^2(\cal D )} \leq \rho_{\theta} \|u-v \|_{L^2(\cal D )},
\end{equation}

for $\rho_{\theta} < 1$. 
\item[4.] There exists a Fredholm Neural Network $\mathcal{F}(\cdot; \cdot, K_{\theta})$, with $M$ hidden layers, such that for any $\epsilon_3 >0$ and all $g \in \mathcal{H}$:
\begin{eqnarray}
    \norm{\mathcal{S}_Kg - \mathcal{F}(\cdot;g, K_{\bf \theta})}_{L^2(\cal D )} < \epsilon_3.
\end{eqnarray}
\end{itemize}
\end{theorem}
\begin{proof}

We prove each of the above statements. 

{\bf 1.}  Recall that, given data for the functions $u$ and $v$, there is not a unique integral kernel $K$ such that $v = T_{K} u$. The choice of a unique kernel satisfying the equation will be achieved using regularization.  If no regularization is employed we can still provide an existence result, where it suffices to construct a kernel that approximates \emph{any of the kernels} satisfying the IE defined by the integral operator. 
\par We assume $K \in  L^2(\mathcal{D} \times \mathcal{D})$,
and recall the class of integral operators:
\begin{equation}
{\mathbb T}:=\{{\cal T}_{K}: {\cal H} \to {\cal H}: \,\, {\cal T}_{K} u(x)=\int_{{\cal D}} K(x,z) u(z) dz \,\,\, \text{for some } K : {\cal D} \times {\cal D} \to {\mathbb R}, \,\, K \in L^2(\mathcal{D} \times \mathcal{D})  \},
\end{equation}
We now consider the set of neural networks:
\begin{equation}
{\cal N} :=\{ N(\cdot , \cdot ; \theta) \,\, : \,\, N(\cdot, \cdot ; \theta) : {\cal D} \times {\cal D} \to {\mathbb R} ,\,\,\, \theta \in \Theta \subset {\mathbb R}^{P}, \,\, P \in {\mathbb N}\},
\end{equation}
which consists of neural networks with inputs $(x,z) \in {\cal D} \times {\cal D}$ parameterized by the weights $\theta \in \Theta$. The structure of the networks is such that $N(\cdot, \cdot ; \theta) \in {\mathcal Z}$ for every $\theta \in \Theta$. Finally, we consider the set of corresponding operators:
\begin{equation}
{\mathbb T}_{\Theta}:= \{ {\cal L}_{\theta} := {\cal T}_{N(\cdot, \cdot ;\theta)} \,\, : \,\, \theta \in \Theta \subset {\mathbb R}^{P}, \,\, P \in {\mathbb N} \,\}.
\end{equation}
Clearly, ${\mathbb T}_{\Theta} \subset {\mathbb T}$. We will show that ${\mathbb T}_{\Theta}$ is dense in ${\mathbb T}$ in the strong operator topology.

We recall the density result of the set of neural networks ${\cal N}$ in the space $L^2(\mathcal{D} \times \mathcal{D})$. This implies that for any $K \in  L^2(\mathcal{D} \times \mathcal{D})$ and any $\epsilon_1 >0$, there exists a neural network $N( \cdot, \cdot ;\theta_{\epsilon}) \in {\cal N}$, such that $K_{\epsilon}(\cdot, \cdot) := N(\cdot, \cdot ; \theta_{\epsilon})$ is an approximation of the kernel function $K$ in $Z$ in the sense that:
\begin{equation}
\| K_{{\theta}_{\epsilon}} - K \|_{L^2(\mathcal{D} \times \mathcal{D})} < \epsilon_1.
\label{eq:Kapprox}
\end{equation}
Based on the above, we obtain the approximated integral operator:
\begin{equation}
    {\cal L}_{\theta_{\epsilon}}:={\cal T}_{K_{\theta_{\epsilon}}} \in {\mathbb T}_{\Theta} \quad u \mapsto  {\cal L}_{\theta_{\epsilon}} u:= {\cal T}_{K_{\theta_{\epsilon}}} u, \quad \mbox{with} \quad {\cal L}_{\theta_{\epsilon}}u(x)= {\cal T}_{K_{\theta_{\epsilon}}} u(x)=\int_{{\cal D}} K_{\theta_{\epsilon}}(x,z) u(z) dz.
\end{equation}

The operator ${\cal T}_{\theta_{\epsilon}} \in {\mathbb T}_{\Theta}$  in turn approximates the operator ${\cal T}_{K} \in {\mathbb T}$. This follows by the Banach-Steinhaus theorem. Indeed,
for any $u \in X$ it holds that:
\begin{equation}\label{AAA}
\| {\cal L}_{\theta_{\epsilon}} u - {\cal T}_{K} u \|_{L^2({\cal D})} \le \| K_{\theta_{\epsilon}} - K \|_{L^2({\cal D} \times {\cal D})} \| u \|_{L^2({\cal D})} \le \epsilon 
'\| u \|_{L^2({\cal D})}.
\end{equation}
This implies that the family of operators $\{ {\cal L}_{\theta_{\epsilon}} \}_{\epsilon >0}$ is pointwise bounded. 

Consider a sequence $(\epsilon_{n} )_{n \in {\mathbb N}}$ such that $\epsilon_n \to 0$, and set $\theta_{n}=\theta_{\epsilon_n}$ for each $n \in {\mathbb N}$. Define the corresponding sequence of neural networks $(N(\cdot, \cdot ; \theta_{n}))_{n \in {\mathbb N}}$ generating the sequence of functions $(K_{n})_{n \in{\mathbb N}} \subset \mathcal Z$ (by setting $K_{n}(\cdot, \cdot) :=N(\cdot, \cdot ; \theta_{n})$), It holds that $K_{n} \to K$ in $L^2(\mathcal{D} \times \mathcal{D})$. Moreover, define the sequence of operators ${\cal L}_{n}={\cal T}_{K_{n}}$. By \eqref{AAA} we have that $\| {\cal L}_{n}u - {\cal T}_{K} \|_{L^2({\cal D})} \to 0$, for $n\rightarrow \infty$. The sequence $({\cal L}_{n}u)_{n \in {\mathbb N}}$ is bounded for all $u \in L^2({\cal D})$ i.e., $({\cal L}_{n})_{n \in {\mathbb N}}$ is pointwise bounded. By the Banach-Steinhaus theorem the sequence $({\cal L}_{n})_{n \in {\mathbb N}}$ is also uniformly bounded, hence it converges in the operator norm topology to ${\cal T}_{K}$ (i.e., $\| {\cal L}_{n} - {\cal T}_{K} \|_{{\cal L}(X, X)} \to 0$ as $n \to \infty$).

\par We can now prove the convergence of the solution operator (\ref{eq:SolOperCOnv}).  
Let $K_{\theta}$ be such that $\mathcal{T}_{K_{\theta}}$ approximates $\mathcal{T}_K$. Furthermore, we have assumed that the true solution operator $\mathcal{S}_K$ exists, i.e., $I-\mathcal{T}_K$ is invertible. Therefore, by the result above, we can choose ${\epsilon}_1$ such that: 
\begin{equation}\label{star}
    \| (I-\mathcal{T}_K) - (I - \mathcal{T}_{K_{\theta}}) \| = \| \mathcal{T}_K - \mathcal{T}_{K_{\theta}} \| < {\epsilon}_1 < \frac{1}{\|(I - \mathcal{T}_K)^{-1} \| } 
\end{equation}

We write $I - \mathcal{T}_{K_{\theta}}=(I - \mathcal{T}_K) - (\mathcal{T}_{K_{\theta}} - \mathcal{T}_K)$. Then, by we factorize to obtain:
\begin{equation}
I - \mathcal{T}_{K_{\theta}}=
(I - \mathcal{T}_K)(I - C),
\end{equation}
where we define the operator $C := (I -\mathcal{T}_K)^{-1}(\mathcal{T}_{K_{\theta}} - \mathcal{T}_K).$ By \eqref{star}, we have:
\begin{equation}
\|C\|
\le
\|(I - \mathcal{T}_K)^{-1}\| \, \|\mathcal{T}_{K_{\theta}} - \mathcal{T}_K\| < 1.
\end{equation}
 Hence, $C$ is a contraction and therefore $I - C$ is invertible by the Neumann series, $(I - C)^{-1} = \sum_{n=0}^{\infty} C^n$, with:
\begin{equation}
\|(I - C)^{-1}\|
\le
\frac{1}{1 - \|C\|}.
\end{equation}

From the factorization we deduce
\begin{equation}
(I - \mathcal{T}_{K_{\theta}})^{-1}
=
(I - C)^{-1}(I - \mathcal{T}_K)^{-1}.
\end{equation}

Hence
\begin{equation}\label{ineq1}
\|(I - \mathcal{T}_{K_{\theta}})^{-1}\|
\le
\frac{\|(I - \mathcal{T}_K)^{-1}\|}
{1 - \|(I - \mathcal{T}_K)^{-1}\| \, \|\mathcal{T}_{K_{\theta}} - T_K\|}.
\end{equation}

We now compute the difference of the inverses. Using the resolvent identity, we obtain:
\begin{equation}
S_K - S_{K_{\theta}} = (I - \mathcal{T}_K)^{-1} - (I - \mathcal{T}_{K_{\theta}})^{-1}
=
(I - \mathcal{T}_K)^{-1}
(\mathcal{T}_K - \mathcal{T}_{K_{\theta}})
(I - \mathcal{T}_{K_{\theta}})^{-1}.
\end{equation}

Taking norms, we obtain
\begin{equation}
\|S_K - S_{K_{\theta}}\|
\le
\|(I - \mathcal{T}_K)^{-1}\|
\, \|\mathcal{T}_K - \mathcal{T}_{K_{\theta}}\|
\, \|(I - \mathcal{T}_{K_{\theta}})^{-1}\|,
\end{equation}
and substituting  \eqref{ineq1} we get
\begin{equation}
\|S_K - S_{K_{\theta}}\|
\le
\frac{ \|(I - \mathcal{T}_K)^{-1}\|^2}{1 - \|(I - \mathcal{T}_K)^{-1}\| \, \|\mathcal{T}_K - \mathcal{T}_{K_{\theta}}\|}
\, \|\mathcal{T}_K - \mathcal{T}_{K_{\theta}}\| \leq \frac{\epsilon_1}{1 - \epsilon_1 \|(I-\mathcal{T}_K)^{-1}\|} =: \epsilon_2.
\end{equation}
As $\epsilon_1 \rightarrow 0$, this bound approaches $\epsilon_2 \rightarrow \epsilon_1$, and the result (\ref{eq:SolOperCOnv}) follows.

{\bf 2.} We can now prove  that  the learned operator $\mathcal{T}_{K_{\theta}}$  is a contraction in the case where the corresponding integral operator is a contraction. For linear operators, this is equivalent to $\| \mathcal{T}_{K_{\theta}}\|_{L^2({\cal D})} := \rho_{\theta} <1$. By assumption, we already have $\| T_K\|_{L^2({\cal D})} = \rho < 1$. Then, for any $u \in \mathcal{X}$: 
\begin{flalign}
    \| \mathcal{T}_{\theta}u \|_{L^2({\cal D})} = \| \mathcal{T}_{K_{\theta}}u - \mathcal{T}_K u + \mathcal{T}_K u \|_{L^2({\cal D})} \leq \mathcal{T}_{K_{\theta_{\epsilon}}} u - \mathcal{T}_{K}u\|_{L^2({\cal D})} + \| \mathcal{T}_K u \|_{L^2({\cal D})} < (\epsilon + \rho)\|u\|_{L^2({\cal D})}.
\end{flalign}
We can select $\epsilon < 1-\rho$, so that we obtain $\| \mathcal{T}_{K_{\theta}}\|_{L^2({\cal D})}\|u\|_{L^2({\cal D})} < \|u\|_{L^2({\cal D})}$, i.e., $\|\mathcal{T}_{K_{\theta}}\|_{L^2({\cal D})} := \rho_{\theta} < 1$, as required.

\par {\bf  3.}  It now remains to show that the solution operator $\mathcal{S}_K$ can be approximated by the learned integral operator, and, by extension, the Fredholm Neural Network of a given depth $M$. 
Because the true operator is contractive,
the inverse $(I-\mathcal{T})^{-1}$ exists and we can use the Neumann series to represent it, which in turn gives us the solution operator as:
\begin{equation}
  \mathcal{S}_Kg=(I-\mathcal{T}_K)^{-1}g= \sum_{i=0}^\infty \mathcal{T}_K^{(i)} g.
\end{equation}
This converges absolutely in $L^2$, in the sense that the numerical series $\sum_{i=0}^\infty \| \mathcal{T}_K^{(i)} g \|_{L^2({\cal D})}$ converges.
 For an approximation $K_{\theta}$, we have
\begin{equation}
  \mathcal{S}_{K_{\theta}}g=(I-\mathcal{T}_{K_{\theta}})^{-1}g= \sum_{i=0}^\infty \mathcal{T}_{K_{\theta}}^{(i)} g,
\end{equation}

Now, we can consider the error in the Fredholm Neural Network with $M$ hidden layers and defined on a grid with $N$ nodes, denoted $\mathcal{F}(\cdot; g, K_{\theta})$, which approximates the solution operator:
\begin{eqnarray}
    \norm{\mathcal{S}_Kg - \mathcal{F}(\cdot;g, K_{\bf \theta})}_{L^2(\cal D )} =\norm{\mathcal{S}_K g - \mathcal{S}_{K_{\theta}} g + \mathcal{S}_{K_{\theta}} g-\mathcal{F}(\cdot;g, K_{\bf \theta})}_{L^2(\cal D )}\leq \notag \\
    \norm{\mathcal{S}_K g - \mathcal{S}_{K_{\theta}} g}_{L^2(\cal D )} +\norm{\mathcal{S}_{K_{\theta}} g-\mathcal{F}(\cdot;g, K_{\bf \theta})}_{L^2(\cal D )}
\end{eqnarray}
For the second term, we have from Theorem \ref{FUA} that there exists a Fredholm NNs such that this bound is as small as we want, as the number of layers $M$ increases. For the first term, considering \eqref{eq:supremuminequality}, we have:
\begin{eqnarray}
 \norm{\mathcal{S}_K g - \mathcal{S}_{K_{\theta}} g}_{L^2(\mathcal{D})} 
 \leq \sum_{i=0}^{\infty} \| \big(\mathcal{T}^{(i)}_K-\mathcal{T}^{(i)}_{K_{\theta}}\big) g\|_{L^2(\mathcal{D})} \leq \|g\|_{L^2(\mathcal{D})} \sum_{i=0}^{\infty} \| K^i-{K^{i}_{\theta}}\|_{L^2(\mathcal{D} \times \mathcal{D})}
 \label{eq:normsoperatordiff}
\end{eqnarray}

Then, using  the property for the difference of powers:
\begin{eqnarray}
    K^i-K^{i}_{\theta} = \sum_{j=0}^{i-1} K^{i-1-j} (K-K_{\theta}) K^{j}_{\theta},
\end{eqnarray}
we get (omitting for simplicity the underscores for the norms):
\begin{eqnarray}\label{usefull-identity}
    \| K^i-K^{i}_{\theta}\| \leq \sum_{j=0}^{i-1} \| K^{i-1-j}\| \| (K-K_{\theta}) \| \|K^{j}_{\theta}\|,
\end{eqnarray}
and considering that $\|K\|, \|K_{\theta}\| < \rho_m := \min \{\rho, \rho_{\theta}\} <1$, we get:
\begin{eqnarray}\label{usefull-identity2}
    \| K^i-K^{i}_{\theta}\|_{L^2(\mathcal{D} \times \mathcal{D})} \leq \epsilon_1 \sum_{j=0}^{i-1} {\rho}_m^{i-1}=\epsilon_1 i \rho_m^{i-1},
\end{eqnarray}
Plugging the above in (\ref{eq:normsoperatordiff}), we finally get:
\begin{equation}
    \norm{\mathcal{S}_K g - \mathcal{S}_{K_{\theta}} g}_{L^2(\cal D )} 
  \leq \|g\|_{L^2(\cal D )} \sum_{i=0}^{\infty} \| K^i-K^{i}_{\theta}\|_{L^2(\mathcal{D} \times \mathcal{D})} \leq \epsilon_1  \|g\| \sum_{i=0}^{\infty} i\rho_m^{i-1}=\frac{\epsilon_1  \|g\|_{L^2(\mathcal{D})}}{(1-\rho_m)^2}.
\end{equation}
Hence, we obtain:
\begin{equation}
\norm{\mathcal{S}_Kg - \mathcal{F}(\cdot;g, K_{\bf \theta})}_{L^2(\cal D )} \leq \frac{\epsilon_1 \|g\|_{L^2(\mathcal{D})}}{(1-\rho_m)^2}+\norm{\mathcal{S}_{K_{\theta}} g-\mathcal{F}(\cdot;g, K_{\bf \theta})}_{L^2(\mathcal{D})}
\end{equation}

The error in the Fredholm NN approximation arises from two sources: the use of $M$ hidden layers in the network and the $N$ nodes corresponding to the integral discretization. Hence, we can write:
\begin{equation}
    \mathcal{S}_{K_{\theta}} g- \mathcal{F}(\cdot;g, K_{\bf \theta})=\mathcal{S}_{K_{\theta}} g- \sum_{i=0}^M \mathcal{T}_{K_{\theta}}^{(i)} g + \sum_{i=0}^M \mathcal{T}_{K_{\theta}}^{(i)} g  -\mathcal{F}(\cdot;g, K_{\bf \theta}),
    \end{equation}
    and taking the norm, we have 
    \begin{equation}
        \norm{\mathcal{S}_{K_{\theta}} g-\mathcal{F}(\cdot;g, K_{\bf \theta})}_{L^2(\mathcal{D})} \leq \norm{ \mathcal{S}_{K_{\theta}} -\sum_{i=0}^M \mathcal{T}_{K_{\theta}}^{(i)} g}_{L^2(\mathcal{D})} +\norm{\sum_{i=0}^M \mathcal{T}_{K_{\theta}}^{(i)} g-\mathcal{F}(\cdot;g, K_{\bf \theta})}_{L^2(\mathcal{D})}.
    \end{equation}
    
For the first term, recall that $\|K_{\theta}\|<\rho_m<1$, and we have
\begin{flalign}
    \norm{\mathcal{S}_{K_{\theta}} g-\sum_{i=0}^M \mathcal{T}_{K_{\theta}}^{(i)} g}_{L^2(\mathcal{D})} =\norm{\sum_{i=0}^\infty \mathcal{T}^{(i)}_{K_{\theta}} g-\sum_{i=0}^M \mathcal{T}_{K_{\theta}}^{(i)} g}_{L^2(\mathcal{D})} = \norm{\sum_{i=M+1}^\infty \mathcal{T}^{(i)}_{K_{\theta}} g}_{L^2(\mathcal{D})} \notag \\ \leq \sum_{i=M+1}^\infty \|K_{\theta}\|_{L^2(\mathcal{D} \times \mathcal{D})}^{i} \|g\|_{L^2(\mathcal{D})} =\frac{\rho_m^{M+1}}{1-\rho_m}\|g\|_{L^2(\mathcal{D})}.
    \label{eq:boundfirst}
\end{flalign}
The second term captures the error induced by the discretization. For clarity, below let $K_{{\theta}, N}$ represent the discretized kernel operator used in the Fredholm Neural Network. Then, we obtain:
    \begin{equation}
        \norm{\sum_{i=0}^M \mathcal{T}_{K_{\theta}}^{(i)} g-\mathcal{F}(\cdot;g, K_{\bf \theta})}_{L^2(\mathcal{D})} =\norm{\sum_{i=0}^M \big(K^i_{\theta}-K^i_{\theta,N}\big) g}_{L^2(\mathcal{D})} \leq \|g\|_{L^2(\mathcal{D})} \sum_{i=0}^M \norm{K^i_{\theta}-K^i_{\theta,N}}_{L^2(\mathcal{D} \times \mathcal{D})}.
    \end{equation}
Now assuming that:
\begin{equation}
    \norm{K_{\theta}-K_{\theta,N}}_{L^2(\mathcal{D} \times \mathcal{D})}=\delta(N),
\end{equation}
and taken that 
\begin{equation}
    \sum_{i=0}^M \norm{K^i_{\theta}-K^i_{\theta,N}}_{L^2(\mathcal{D} \times \mathcal{D})} <  \sum_{i=0}^\infty \norm{K^i_{\theta}-K^i_{\theta,N}}_{L^2(\mathcal{D} \times \mathcal{D})}
\end{equation}
we can use \eqref{usefull-identity} and \eqref{usefull-identity2} to get: 
\begin{equation}
    \norm{\sum_{i=0}^M \mathcal{T}_{K_{\theta}}^{i} g-\mathcal{F}(\cdot;g, K_{\bf \theta})} \leq \|g \|_{L^2(\mathcal{D})} \delta(N) \sum_{i=0}^{\infty} i \rho_m^{i-1} \leq \frac{\delta(N)  \|g\|_{L^2(\mathcal{D})}}{(1-\rho_m)^2}
\end{equation}
and
\begin{equation}
    \norm{\mathcal{S}_Kg - \mathcal{F}(\cdot;g, K_{\bf \theta})} \leq \frac{\epsilon_1  \|g\|_{L^2(\mathcal{D})}}{(1-\rho_m)^2} + \frac{\rho_m^{M+1}}{1-\rho_m}\|g\|_{L^2(\mathcal{D})} + \frac{\delta(N)  \|g\|_{L^2(\mathcal{D})}}{(1-\rho_m)^2}.
\end{equation}
As $M,N$ increase $\rho_m^M, \delta(N) \rightarrow 0$, and therefore, for any $\epsilon_3 >0$, we can select $M,N$ and $\epsilon_1$ such that 
\begin{equation}
    \norm{\mathcal{S}_Kg - \mathcal{F}(\cdot;g, K_{\bf \theta})}_{L^2(\mathcal{D})} \leq \epsilon_3,
\end{equation}
as required.
\color{black}
\end{proof}


We now turn to the non-linear integral operator case. We formulate a similar theorem to the UAT above, based on the notion of that we can approximate both the kernel $K$ and non-linearity $G$, from which the operator is defined. Given the building blocks for the linear case, we formulate the following result directly using the Recurrent Fredholm NN approximation. 

\begin{theorem}[FREDINOs as universal approximators of non-linear integral operators] Let $K:\mathcal{D}\times \mathcal{D} \rightarrow \mathbb{R}$ and $G:\mathcal{D} \rightarrow \mathcal{D}$ be a pair of kernel and non-linearity satisfying the integral equation $f(x) = g(x) + (\mathcal{T}_{K,G}f)(x)$, where the operator ${\cal T}_{K,G} : {\cal H} \to {\cal H}, (\mathcal{T}_{K,G})u = \int_{\mathcal{D}} K(x,z)G(u(z))dz$, is a contraction with contraction constant $\rho <1$. Furthermore, suppose that:
\begin{flalign}
    C_K:=\sup_{x \in \mathcal{D}} \int_{\mathcal{D}} |K(x,y)|dy < \infty, \hfill \text{ and} \\
    |G(x) - G(y)| \leq L_G|x-y|,
\end{flalign}
for ally $x,y \in \mathcal{D}$ and assume that $C_KL_G < 1$. Finally, let $\mathcal{S}_{K,G}$ represent the solution operator of the integral equation $f(x) = g(x) + (\mathcal{T}_{K,G}f)(x)$, i.e., $(\mathcal{S}_{K,G})g =f$. Then, for $\epsilon >0$, there exist neural networks $K_{ \theta}$ and $G_{\vartheta}$, for some parameter set $\tilde{\theta} = \{\theta, \vartheta\} \in \Theta \subset \mathbb{R}^P, P \in \mathbb{N}$, and a Recurrent Fredholm NN $\mathcal{F}(\cdot; \cdot, K_{\theta},G_{\vartheta})$, with $M$ hidden layers, such that:
\begin{eqnarray}
    \norm{\mathcal{S}g - \mathcal{F}(\cdot;g, K_{\bf \theta}, G_{\vartheta})}_{L^2(\cal D)} < \epsilon,
\end{eqnarray}
and the learned operator:
\begin{equation}
    (\mathcal{T}_{\tilde{\theta}}u)(x) := \int_{\mathcal{D}} K_{\theta}(x,z) G_{\vartheta}\big(u(z)\big) d z,
\end{equation}

is contractive, i.e., 
\begin{equation}
    \| \mathcal{T}_{\tilde{\theta}} u - \mathcal{T}_{\tilde{\theta}} v \|_{L^2(\cal D)} \leq \rho_{\tilde{\theta}} \|u-v\|_{L^2(\cal D)},
\end{equation}
for $\rho_{\tilde{\theta}} < 1$.
\end{theorem}
\begin{proof}
As above, we will use the density result of the  set of neural networks ${\cal N}$ in the space $L^2(\mathcal{D} \times \mathcal{D})$. This will allow us to approximate both $K$ and $G$. Specifically, we have that for any $K \in L^2(\mathcal{D} \times \mathcal{D})$ and any $\epsilon >0$, there exists a neural network $N( \cdot, \cdot ;\theta \in \Theta)$, such that $K_{\theta}(\cdot, \cdot) := N(\cdot, \cdot ; \theta)$ is an approximation of the kernel function $K$ in $\mathcal Z$ in the sense that:
\begin{equation}
\| K_{\theta} - K \|_{L^2(\cal D \times \cal D)} < \epsilon_K,
\end{equation}
for any $\epsilon_K > 0$. Similarly, for any $G$, there exists network $N(\cdot; \vartheta \in \Theta)$ such that 
\begin{equation}
    \| G_{\vartheta} - G \|_{L^2(\cal D)} < \epsilon_G,
\end{equation}
for any $\epsilon_G > 0$ and where $G_{\vartheta} = N(\cdot; \vartheta)$. The above are guaranteed by the argument in Theorem \ref{theorem}. This now allows us to consider bounds for the approximations of $G$ and $K$. 
\par We start by assuming that there exists constant $L_{G_{\vartheta}}$ such that $|G_{\vartheta}(x) - G_{\vartheta}(y)| \leq L_{G_{\vartheta}}|x-y|$. Then, consider $u,v \in \mathcal{X}$ and, for any $x \in \mathcal{D}$: 
    \begin{flalign}
        |(\mathcal{T}_{\tilde{\theta}}u)(x) - & (\mathcal{T}_{\tilde{\theta}} v)(x) |  = \left|\int_{\mathcal{D}} K_{\theta}(x,z)G_{\vartheta}(u(z))dz -  \int_{\mathcal{D}} K_{\theta}(x,z)G_{\vartheta}(v(z)) dz\right| \notag \\ 
        &\leq  \int_{\mathcal{D}}|K_{\theta}(x,z)|L_{G_\vartheta}|u(z) - v(z)| dz  \leq L_{G_\vartheta}\|u-v\|\int_{\mathcal{D}}|K_{\theta}(x,z)|dz =: L_{G_{\vartheta}}C_{K_\theta}\| u - v\|_{L^2(\cal D)}, 
    \end{flalign}
from which we have $\|\mathcal{T}_{\tilde{\theta}}u -  \mathcal{T}_{\tilde{\theta}}v\|_{L^2(\cal D)} \leq L_{G_{\vartheta}}C_{K_\theta}\|u-v\|_{L^2(\cal D)}.$

Hence, it suffices to show that $L_{G_{\vartheta}}C_{K_\theta}=: \rho_{\tilde{\theta}} < 1$. We can do this by comparing to the corresponding constants $L_G$ and $C_K:= \int_{\mathcal{D}} |K(x,z)|dz$. To this end, assuming that $\|K_{\theta} - K\|_{L^2(\mathcal{D})} < \epsilon_K$, for any $x\in \mathcal{D}$ we have: 
\begin{flalign}\label{eq1}
 \left| \int_{\mathcal{D}} |K_{\theta}(x,z)|dz - \int_{\mathcal{D}} |K(x,z)|dz \right| &\leq \int_{\mathcal{D}}\left|\left|K_{\theta}(x,z)\right| - \left|K(x,z)\right| \right| dz \notag \\
 &\leq \int_{\mathcal{D}}|K_{\theta}(x,z) - K(x,z)|dz < c \epsilon_K,
\end{flalign}
where $c > 0$ is the volume of $\mathcal{D}$, and therefore: $\| C_{K_\theta} - C_K\|_{L^2({\cal D})} < C\epsilon_K$. Since $\epsilon_K \rightarrow 0$, we can deduce that $C_K \rightarrow C$. 

\par We now turn to the approximation $G_{\vartheta}$, for which we can write $\|G_{\vartheta} - G\|_{L^2(\mathcal{D})} < \epsilon_G$. We want to show that the Lipschitz constant $L_{\vartheta}$ is arbitrarily close to $L_G$. This follows as shown below:
\begin{flalign}\label{eq2}
    \left|G_{\vartheta}(x) - G_{\vartheta}(y) \right| \leq \left|G_{\vartheta}(x) - G(x) \right| + \left|G(x) - G(y) \right| + \left|G(y) - G_{\vartheta}(y) \right|  \leq 2\epsilon_G + L_G|x-y|.
\end{flalign}
As above, we know that there exists a choice of  parameters $\vartheta$, such that $\epsilon_G \rightarrow 0$, and therefore $L_{G_{\vartheta}} \rightarrow L_G$. From \eqref{eq1} and \eqref{eq2}, it follows that: 
\begin{equation}
    \| \mathcal{T}_{\tilde{\theta}} u -  \mathcal{T}_{\tilde{\theta}}v\|_{L^2(\mathcal{D})}  \leq \rho_{\tilde{\theta}}\|u - v\|_{L^2(\mathcal{D})} \rightarrow L_{G}C_K\|u-v\|_{L^2(\mathcal{D})}, 
\end{equation}
and hence $\rho_{\tilde{\theta}} < 1$, since $C_K L_G<1$, by assumption.



We now turn to the solution operator. Since we are in the non-linear case, note that we have:
\begin{equation}
S_{{K, G}} \, g = \lim_{n \to \infty} \mathcal{T}^{(n)} g \quad \text{and} \quad S_{K_{\theta}, G_{\vartheta}} \, g = \lim_{n \to \infty} \mathcal{T}_{K_{\theta}, G_{\vartheta}}^{(n)} \, g.
\end{equation}
For brevity, we will use the notation $S \equiv S_{K,G}$ and $T_{\tilde{\theta}} \equiv  T_{K_{\theta}, G_{\vartheta}}, S_{\tilde{\theta}} \equiv  S_{K_{\theta}, G_{\vartheta}}$, and we also omit the norm subscripts. We now want to prove that $S_{\tilde{\theta}} g \to S_{\mathcal{G}} g$. To this end, we can consider an inductive argument. Consider first the case where $n=1$:
\begin{align}
\|\mathcal{T}_{\tilde{\theta}} u - &\mathcal{T} u\| = \left\| \int_{\mathcal{D}} K_\theta(x,z) G_{\vartheta}(u(z)) dz - \int K(x,z) G(u(z)) dz \right\| \nonumber\\
&\leq \left| \int_{\mathcal{D}} K_\theta(x,z) G_{\vartheta}(u(z)) dz - \int_{\mathcal{D}} K_\theta(x,z) G(u(z)) dz + \int_{\mathcal{D}} K_\theta(x,z) G(u(z)) dz - \int_{\mathcal{D}} K(x,z) G(u(z)) dz \right| \nonumber\\
&\leq \int_{\mathcal{D}} |K_\theta(x,z)| |G_{\vartheta}(u(z)) - G(u(z))| dz + \int_{\mathcal{D}} |G(u(z))| |K_\theta(x,z) - K(x,z)| dz.
\end{align}
Hence, for any $u \in \mathcal{H}$:
\begin{equation}\label{nl-bound}
\|\mathcal{T}_{\tilde{\theta}} u - \mathcal{T} u\| \leq \tilde{\epsilon}:= C_{K_{\theta}}  \epsilon_G + m \varepsilon_K,
\end{equation}
where $m:= \sup_{x \in \mathcal{D}} |G(u(x))|$.

We now consider the case that  $n=2$. Let $u' := \mathcal{T} u$ and $u'' := \mathcal{T}_{\tilde{\theta}} u$. Then:
\begin{align}
\|\mathcal{T}^{(2)}_{\tilde{\theta}}u - \mathcal{T}^{(2)}u \| = \|\mathcal{T}_{\tilde{\theta}} u'' - \mathcal{T} u'\|  =& \|\mathcal{T}_{\tilde{\theta}} u' - \mathcal{T} u' + \mathcal{T} u' - \mathcal{T} u''\| 
\leq  \|\mathcal{T}_{\tilde{\theta}} u' - \mathcal{T} u'\| + \|\mathcal{T} u' - \mathcal{T} u''\| \nonumber\\
&\leq \tilde{\epsilon} + \rho \|u' - u''\| \leq \tilde{\epsilon} + \rho \tilde{\epsilon},
\end{align}
where we have used \eqref{nl-bound} and the fact that $\mathcal{T}$ is a contraction.
It easily follows that, for any $n\geq 2$:
\begin{equation}
\|\mathcal{T}_{\tilde{\theta}}^{(n)} u - \mathcal{T}u^{(n)}\| \leq \sum_{i=0}^{n-1} \tilde{\epsilon} \rho^i,
\end{equation}
and as $n \to \infty$:
\begin{equation}\label{nl-bound-2}
\|S_{{\tilde{\theta}}}g - Sg\| = \| \lim_{n \rightarrow \infty}\mathcal{T}_{{\tilde{\theta}}}^{(n)}u - \lim_{n \rightarrow \infty}\mathcal{T}^{(n)}u \| =  \lim_{n \rightarrow \infty} \|\mathcal{T}_{\tilde{\theta}}^{(n)} u - \mathcal{T}u^{(m)}\| \leq \frac{\tilde{\epsilon}}{1 - \rho}.
\end{equation}

Finally, we have to show that the (Recurrent) Fredholm Neural Network, $\mathcal{F}(\cdot; g, K_{\theta}, G_{\vartheta})$, with $M$ hidden layers, can approximate the solution operator arbitrarily close. We have: 
\begin{align}
\|S g - \mathcal{F}(\cdot; g, K_\theta, G_\vartheta)\| 
&= \|S g - S_{\Theta} g + S_{{\tilde{\theta}}} g - \mathcal{F}(\cdot; g,  K_\theta, G_\vartheta)\| \nonumber\\
&\leq \|S g - S_{{\tilde{\theta}}} g\| + \|S_{{\tilde{\theta}}} g - \mathcal{F}(\cdot; g,  K_\theta, G_\vartheta)\|.
\end{align}

For the first term we have the bound given by \eqref{nl-bound-2}. For the second term, we get:
\begin{align}\label{nl-bound-3}
\|S_{{\tilde{\theta}}} g - \mathcal{F}(\cdot; g, K_\theta,  G_\vartheta)\| 
&= \|S_{{\tilde{\theta}}} g - \mathcal{T}_{{\tilde{\theta}}}^{(M)} g + \mathcal{T}_{{\tilde{\theta}}}^{(M)} g - \mathcal{F}(\cdot; g, K_\theta, G_\vartheta)\| \nonumber\\
&\leq \|S_{K_\theta, G_\theta} g - \mathcal{T}_{{\tilde{\theta}}}^{(M)} g\| + \|\mathcal{T}_{{\tilde{\theta}}}^{(m)} g - \mathcal{F}(\cdot; g, K_\theta, G_\vartheta)\| \nonumber\\
&\leq \frac{\rho^M}{1-\rho} \|\mathcal{T}_{{\tilde{\theta}}} g - g\| + {\delta(M)}\frac{1 - \rho^M}{1-\rho},
\end{align}
where in the last step we have used the error bound for the $M-$th iteration in the fixed point estimation and the error bound induced by the integral discretization in the Recurrent Fredholm NN.

\par By combining \eqref{nl-bound-2} and \eqref{nl-bound-3} we obtain the error bound when approximating the true solution operator with the Recurrent Fredholm NN. It follows that, for any $\epsilon >0$, we can choose $\epsilon_K, \epsilon_G$ and the Fredholm NN depth, $M$, such that
\begin{equation}
 \|S g - \mathcal{F}(\cdot; g, K_\theta, G_\vartheta)\|     \leq \frac{\tilde{\epsilon}}{1 - \rho} + \frac{\rho^M}{1-\rho} \|\mathcal{T}_{{\tilde{\theta}}} g - g\| + {\delta(M)}\frac{1 - \rho^M}{1-\rho} < \epsilon,
\end{equation}
completing the proof.
\end{proof}

\color{black}

\subsection{The Optimization Problem}
With the theoretical framework described above, the operator learning problem can be formulated as a training problem, whereby the goal is to minimize appropriately defined loss functions. 
\par We first, consider the linear problem. To clearly state the learning objective, we aim to estimate the unknown kernel $K$, such that for any $\{g, f\}$, $g_i, f_i \in \mathcal{H}$ or $\mathcal{X}$, the FIE $f = g + \mathcal{T}_Kf$ is satisfied. We model the unknown kernel $K_{\theta}$, depending on some parameter set $\theta \in \Theta$. This kernel model is used in the Fredholm NN, denoted $\mathcal{F}(\cdot;g, K_{\theta})$, whose output is the function $f$. Hence, we aim to solve the regularization problem: 
\begin{equation}\label{loss-linear}
\begin{aligned}
\min_{K_{\theta} \in L^{2}({\cal D} \times {\cal D})} \| f - \mathcal{F}(\cdot; g, K_{\theta}) \|_{L^{2}({\cal D})}^{2} + \mathcal{R}(\theta),
\end{aligned}
\end{equation}
where
 \begin{equation}
    \mathcal{R}(\theta) =
     \lambda_{K}\|K_{\theta} \|^2_{L^2({\cal D} \times {\cal D})}. 
 \end{equation}

Similarly, for the non-linear problem, we aim to estimate the unknown kernel $K$ and non-linearity $G$, such that for any $\{g, f\}$, $g_i, f_i \in \mathcal{H}$ or $\mathcal{X}$, the FIE $f = g + \mathcal{T}_{K,G}f$ is satisfied. We now consider the model $K_{\theta}$, as well as a model for the non-linearity $G_{\vartheta}$, with $\vartheta \in \Theta$. The corresponding Recurrent Fredholm NN, $\mathcal{F}(\cdot; g, K_{\theta}, G_{\vartheta})$ is trained by solving the regularization problem:
\begin{equation}\label{loss-non-linear}
\begin{aligned}
\min_{K_{\theta} \in L^{2}({\cal D} \times {\cal D})} \| f - \mathcal{F}(\cdot; g, K_{\theta}, G_{\vartheta}) \|_{L^{2}({\cal D})}^{2} + \mathcal{R}(\theta, \vartheta),
\end{aligned}
\end{equation}
with
\begin{equation}
    \mathcal{R}(\theta, \vartheta) =  \lambda_{K}\|K_{\theta} \|^2_{L^2} + \lambda_{G}\|G_{\vartheta} \|^2_{L^2}.
\end{equation}

\par More details regarding the definition of these loss functions and how the Fredholm NN architecture enhances the training process is given below. However, it is first important to discuss the intuition and interpretation of the training framework described in this section. Fredholm NN consider a specific, fixed relationship between the weights and biases, which is determined by the form of underlying integral operator of the FIE. Here, since these values are unknown we essentially construct a \emph{trainable Fredholm NN}: we are assuming that there exists a Fredholm NN that approximates the solution of the FIE, but whose parameters are unknown (as in classical NN training). 
\par However, given the deterministic relationship between the parameters (e.g., we know that the weights connecting each pair of consecutive hidden layers are the same and correspond to the discretized values of the kernel $K(x,y)$), we are essentially training on a constrained space of possible parameter values. In other words, we are training a deep neural network, but constrained by the mathematical formulation of the FIE, thereby reducing the space of possible parameter values and inducing interpretability in the model. Hence, in practice, this can be achieved by considering a simple model for the unknown parameters, which is much easier to train. Training the simple model therefore becomes equivalent to training the full Fredholm NN (as the Fredholm NN parameters are simply the outputs of the parameter models, $K_{\theta}, G_{\vartheta}$). This is further discussed in the following section.

\subsection{Ensuring the construction of contractive FREDINOs}

In this section we discuss in more detail how the proposed methodology ensures that the learned models indeed define contractive operators.   Recall that the solution to the given operator learning is not unique. To this end, we introduce the regularized problem for the integral operator:
\begin{equation}\label{kalamata-1}
\begin{aligned}
\min_{K \in L^{2}({\cal D} \times {\cal D}), \,\,\, \| K \|_{L^2({\cal D}\times {\cal D})} < 1} \| f - \mathcal{T}_{K} f - g \|_{L^{2}({\cal D})}^{2} + \lambda \| K \|_{L^{2}({\cal D} \times {\cal D})}^{2},
\end{aligned}
\end{equation}
or the equivalent problem for the solution operator
\begin{equation}\label{kalamata_2}
    \begin{aligned}
        \min_{K \in L^{2}({\cal D} \times {\cal D}), \| K \|_{L^2({\cal D}\times {\cal D})} <1} \| f - S_K g \|_{L^{2}({\cal D})}^{2} + \lambda \| K \|_{L^{2}({\cal D} \times {\cal D})}^{2}.
    \end{aligned}
\end{equation}
In both problems \eqref{kalamata-1} and \eqref{kalamata_2} we introduce a regularization term $\lambda \| K \|_{L^{2}({\cal D} \times {\cal D})}^{2}$, which is a Tikhonov type regularization term that allows us to select a unique kernel $K$ out of the infinite set of kernels that may fit our data. Formulation \eqref{kalamata_2} of the problem is precisely suited for our treatment, using Fredholm NNs. since the constraint $\| K \|_{L^2({\cal D}\times {\cal D})} < 1$ forces us to learn a contractive integral operator that can generate the observed training data. This constraint is identical to using the Fredholm NN pass during training, since its output converges only if the operator $\mathcal{T}_{K}$ is a contraction so that the solution operator can be calculated by the Neumann series. Another equivalent way of formulating the optimization problem \eqref{kalamata_2} is as below:
\begin{equation}\label{kalamata-3}
\begin{aligned}
    \min_{K \in L^{2}({\cal D}\times {\cal D})} \| f - S_Kg \|_{L^{2}({\cal D})}^{2} + \lambda \| K \|_{L^2({\cal D}\times {\cal D})}^2 + {\mathbb I}_{\| \cdot \|_{L^2({\cal D}\times {\cal D})} < 1}(K)
\end{aligned}
\end{equation}
where by ${\mathbb I}_{\| \cdot \|_{L^2({\cal D}\times {\cal D})} < 1}(\cdot)$, we denote the convex indicator function of the interior of the unit ball in $L^2({\cal D}\times {\cal D})$, which assumes the value $0$ for any $K$ such that $\| K \|_{L^2({\cal D}\times {\cal D})}<1$ and $+\infty$ otherwise.  

Problem \eqref{kalamata-3} is an infinite dimensional problem in the function space $L^2({\cal D}\times {\cal D})$, that will yield as solution the optimal kernel $K$, satisfying also the contraction property for the corresponding integral operator $\mathcal{T}_{K}$. In practice, problem \eqref{kalamata-3} is replaced by an approximate, finite dimensional problem, that uses in lieu of the whole $L^2({\cal D}\times {\cal D})$ a dense subset ${\cal K}_{\theta} \subset L^2({\cal D}\times {\cal D})$, which is finite dimensional (e.g., a neural network approximation of the kernels $K$ generating the integral operators $\mathcal{T}_{K}$). In fact, we replace problem \eqref{kalamata-3} by the finite dimensional approximation
\begin{equation}\label{kalamata-4}
\begin{aligned}
    \min_{K_{\theta} \in {\cal K}_{\theta}\subset L^{2}({\cal D}\times {\cal D})} \| f- S_K g -f \|_{L^{2}({\cal D})}^{2} + \lambda \| K _{\theta}\|_{L^2({\cal D}\times {\cal D})}^2 + {\mathbb I}_{\| \cdot \|_{L^2({\cal D}\times {\cal D})} < 1}(K_{\theta}).
\end{aligned}
\end{equation}
Given that the Fredholm NN  blows up if the associated kernel satisfies $\| K_{\theta}\|_{L^2({\cal D}\times {\cal D})} \ge 1$, we see that the set where $S_{K_{\theta}}g = {\cal F}(\cdot ; g, K_{\theta})$ (where ${\cal F}(\cdot ; g, K_{\theta})$ is the Fredholm NN generated by the approximate kernel $K_{\theta}$), coincides with the set where ${\mathbb I}_{\| \cdot \|_{L^2({\cal D}\times {\cal D})} < 1}(K_{\theta})=0$. Hence, problem \eqref{kalamata-4} in terms of kernel approximations $K_{\theta} \in {\cal K}_{\theta}$, coincides with problem
\begin{equation}\label{kalamata-5}
\begin{aligned}
\min_{ K_{\theta} \in {\cal K}_{\theta}} \| f-  {\cal F}(\cdot; g, K_{\theta}) \|_{L^{2}({\cal D})}^2 + \lambda \| K_{\theta} \|_{L^2({\cal D}\times {\cal D})}^2,
\end{aligned}
\end{equation}
as provided in the aforementioned loss functions.

To summarize, our goal is to solve problem \eqref{kalamata-3}, which is approximated by the finite dimensional problem \eqref{kalamata-4} which, as shown above, is equivalent to problem \eqref{kalamata-5}, formulated in terms of Fredholm NNs. This shows how we can use the custom Fredholm NN architecture to ensure that our results adhere to the contractivity requirement, i.e., the forward pass itself it equivalent to solving an optimization problem, but within the constrained space of contractive operators. 

\par An equivalent argument holds for the non-linear FIE problem. The schematic representation of the proposed framework is given in Fig. \ref{fig:fred_No}.
\begin{figure}
    \centering
    \includegraphics[width=1.1\textwidth]{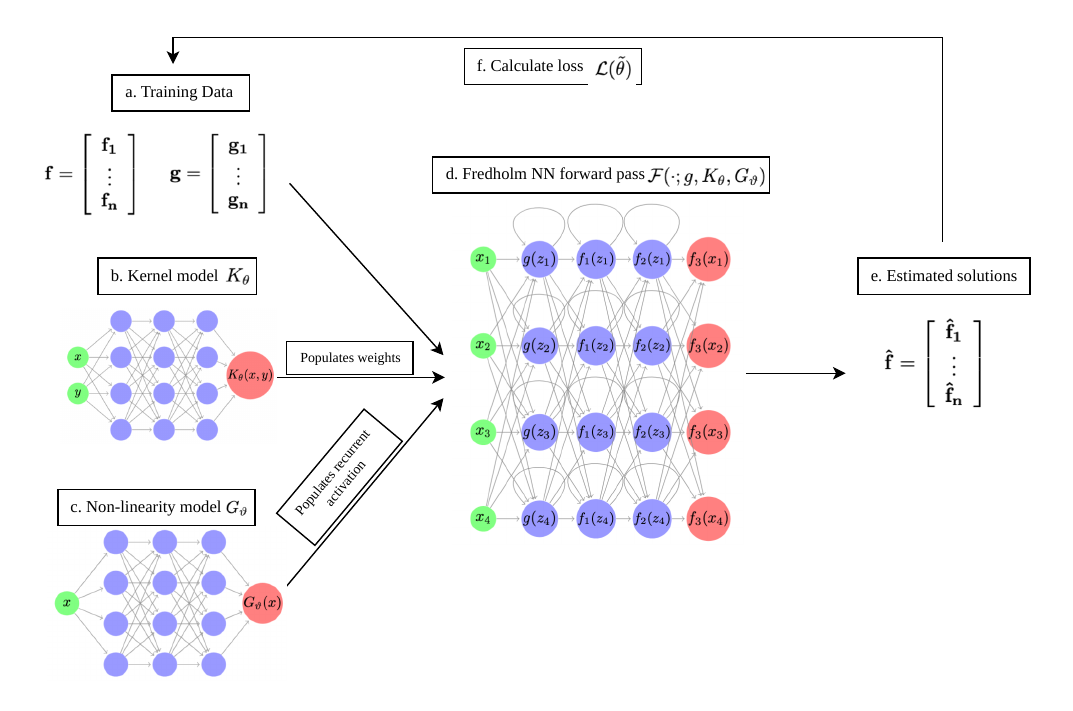}
    \caption{The Fredholm Neural Operator (FREDINO) framework: The training data consists of a family of functions $g_i$ and the corresponding solutions to the FIE $f_i$, for $i=1,\dots, n$. We model the unknown integral kernel and the non-linearity with two neural networks $K_{\theta}(x,y)$ and $G_{\vartheta}(x)$. These models then populate the parameters of the Fredholm Neural Network. The loss function $\mathcal{L}(\Theta)$ is calculated by comparing output of the populated Fredholm NN, $\hat{f}_i, i=1,\dots,n$ with the training data. For brevity, the schematic refers to the case of the non-linear FIE, using the Recurrent Fredholm NN. The linear case is simpler, where we only consider $K_{\theta}$.}
    \vspace{-0.5cm}
    \label{fig:fred_No}
\end{figure}

\section{Learning Elliptic PDEs with the Fredholm Integral Neural Operators}\label{pde-sec}
Finally, we show how FREDINOS can be employed in the setting of elliptic PDEs. Given the connection between integral operators and the general theory of PDEs (see \cite{brebbia2016boundary}), the proposed approach can find applications in many similar domains. In this paper, we focus on the case of semi-linear elliptic PDEs, continuing from our previous work in \cite{georgiou2025fredholm2} where the application of Fredholm NNs to various elliptic PDEs was rigorously established, utilizing Potential Theory and the Boundary Integral Equation (BIE) method, in order to develop the \emph{Potential Fredholm Neural Network} (PFNN), which replicates the double layer potential formulation (also see \cite{deng1996boundary, sakakihara1987iterative}). Specifically, the PFNN consists of a Fredholm NN for the solution of the BIE and an extra layer and a final hidden layer that replicates the double layer potential integral to obtain the solution to the PDE.
A schematic of this is given in Fig. \ref{fig:pfnn}, and the mathematical representation for the semi-linear case is given below. Note that, for the present work, we consider specifically the 2D PDE example, in line with the methodology developed in \cite{georgiou2025fredholm, georgiou2025fredholm2}. Higher dimensional cases require some different handling due to the corresponding fundamental solutions for $d\geq 3$, and will be considered in subsequent work, as the goal of the present paper is to develop the core methodology of the Fredholm Neural Operator.
\begin{figure}
    \centering
    \includegraphics[width=0.65\textwidth]{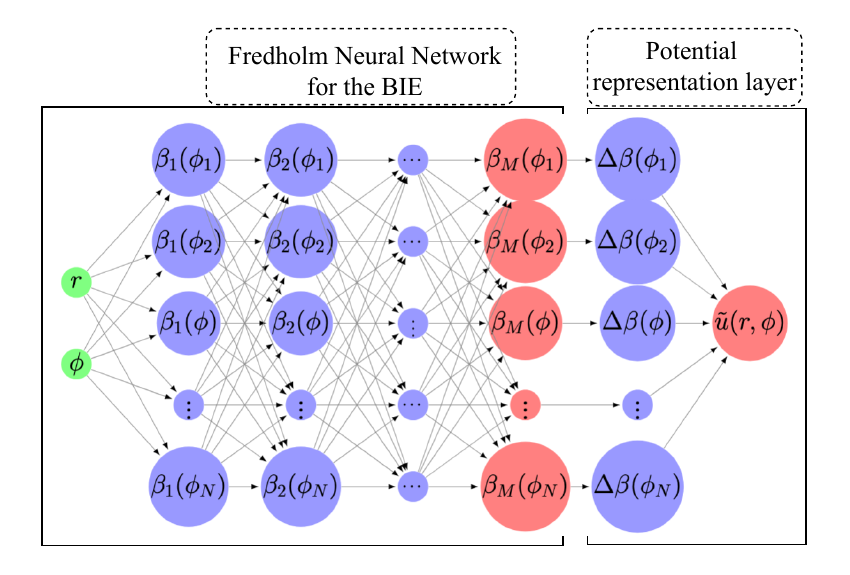}
    \caption{The Potential Fredholm Neural Network (PFNN) for the solution of 2D linear elliptic PDEs: the first component is a Fredholm NN solving the Boundary Integral Equation and the final hidden layer corresponds to the double layer potential formulation of the solution to the PDE, $\tilde{u}(r, \phi)$.}
    \label{fig:pfnn}
    \vspace{-0.5cm}
\end{figure}
\par Consider the 2D non-linear elliptic PDE:
\begin{eqnarray}\label{nl-pde}
\begin{cases}
 \Delta u(x) = \psi(u(x),x), \quad x \in \Omega= \{x,y \in \mathbb{R} : x^2 + y^2 \leq 1 \} \\ 
u(x) = g(x), \quad x \in \partial \Omega,
\end{cases}
\end{eqnarray}
where $\psi:\mathbb{R}^2 \rightarrow \mathbb{R}$ is a a non-linear source term. This can be solved using the fixed point scheme (see approach as described in \cite{georgiou2025fredholm2}):
\begin{eqnarray}\label{nl-pde2}
\begin{cases}
 \Delta u_{n+1}(x) -\lambda u_{n+1}(x)= \tilde{\psi}(u_n(x),x), \quad x \in \Omega\\ 
u(x) = g(x), \quad x \in \partial \Omega,
\end{cases}
\end{eqnarray}
where $\tilde{\psi}(u(x),x) := -\lambda u(x) + \psi(u(x),x)$,
where each iteration is solved using the PFNN. 
\par As an equivalent representation, we can adopt a coupled, implicit integral formulation of the double layer potential and BIE, as follows (see \cite{georgiou2025fredholm2} for proof of the following representation):
\begin{flalign}\label{bie1}
u(x)=\int_{\partial\Omega}(\beta(y)-&\beta(x^*))\left(\frac{\partial\Phi(x,y)}{\partial n_y}- \frac{\partial\Phi(x^*,y)}{\partial n_y}\right)d\sigma_y  +\beta(x^*)\left(\frac12+\int_{\Omega}\lambda\delta\Phi(x,y)dy\right) \notag \\ & +\int_{\Omega}\Phi(x,y)\tilde{\psi}(u(y), y)dy+\int_{\partial\Omega}\beta(y)\frac{\partial\Phi(x^*,y(s))}{\partial n_y}d\sigma_y =:  \mathcal{P}(x; \Phi_{\theta}, \beta, u), \,\,\, x \in \Omega,
\end{flalign}
\begin{flalign}\label{bie2}
g(x) =
\int_{\partial\Omega}
\beta(y)
\frac{\partial \Phi(x,y)}{\partial n_y}
d\sigma_y+\int_{\Omega} \Phi(x,y) \tilde{\psi}(u(y),y)dy + \frac{1}{2}\,\beta(x) =: \mathcal{D}(x;\Phi_{\theta}, \beta, u), \,\,\, x \in \partial\Omega,
\end{flalign}
where $x \in \Omega$ and $x^* \in \partial \Omega$ is the corresponding point on the boundary, i.e., $x = (r \cos(\theta), r \sin(\theta))$ and $x^* = (\cos(\theta), \sin(\theta))$, $\Phi$ is the fundamental solution of the Helmholtz PDE depending on the modified Bessel function of the second kind, $\Phi(x,y) = -\frac{1}{2\pi}K_0(|x-y|)$, and $ \delta \Phi(x,y) := \Phi(x,y) - \Phi(x^*,y)$. 

\par By looking directly at the problem as a non-linear PDE however, the fundamental solution $\Phi$ in standard sense does not exist. Hence, we can rather consider the problem of learning a surrogate model $\Phi_{\theta}$ such that \eqref{bie1} and \eqref{bie2} are satisfied, for any boundary function $g(x)$. This way, we learn the operator $\mathcal{A}:\mathcal{X} \rightarrow \mathcal{X}$, with $\mathcal{X} := \mathcal{C}(\Omega, \mathbb{R})$, such that: 
\begin{eqnarray}\label{nl-pde3}
\begin{cases}
 \mathcal{A} u(x)  = \tilde{\psi}(u(x),x), \quad x \in \Omega\\ 
u(x) = g(x), \quad x \in \partial \Omega,
\end{cases}
\end{eqnarray}
by learning $\Phi_{\theta}$ and the corresponding integral operators that define the coupled integral equations \eqref{bie1}, \eqref{bie2}.
The standard way to construct the loss function to solve this problem would be to combine two terms that incorporate the system of integral equations:
\begin{equation}\label{loss-pfnn}
\begin{aligned}
\min_{\Phi_{\theta} \in L^{2}({\cal D} \times {\cal D})} \| u- \mathcal{P}(\cdot; \Phi_{\theta}, \beta, u) \|_{L^{2}({\cal D})}^{2} + \|g - \mathcal{D}(\cdot;\Phi_{\theta},\beta, u) \|_{L^{2}({\cal D})}^{2}+ \mathcal{R}(\theta),
\end{aligned}
\end{equation} 
the regularization term is given by $\mathcal{R}(\theta) = \lambda_K \| \Phi_{\theta}\|_2^2$. 
\par However, we can instead use the PFNN framework to ensure that the learned model is such that the integral operator in the BIE is a contraction, whilst simultaneously ensuring that the \eqref{bie1} and \eqref{bie2} are satisfied. 
Therefore, we denote the PFNN by $\mathcal{P}^F(\cdot; g, \Phi)$, which approximates the solution $u(x)$, and construct the loss function:
\begin{equation}\label{loss-pfnn-2}
    \min_{\Phi_{\theta} \in L^{2}({\cal D} \times {\cal D})} \| u- \mathcal{P}^F(\cdot; g, \Phi_{\theta}) \|_{L^{2}({\cal D})}^{2} + \mathcal{R}(\theta).
\end{equation}
Notice how this approach essentially absorbs the second term of the loss function \eqref{loss-pfnn}, by enforcing the PFNN architecture in the forward pass during training. This is in line with the notion mentioned above, whereby we construct a \emph{trainable Potential Fredholm NN} whose output matches the solution of the PDE, and its architecture implicitly enforces the satisfaction of the BIE, simultaneously.
  
\section{Numerical Results}\label{sec5}
In this section, we consider various examples to assess the proposed methodology, in particular linear and nonlinear integral operators of Frehdolm type in arbitratu dimensions as well as nonlinear elliptic PDEs in 2D.  Throughout the numerical examples, we base the approach on the form of input functions used in \cite{fabiani2025randonets}. Specifically, as inputs $g$ we use a superposition of a Gaussian mixture, a cubic polynomial, and sinusoids:
\begin{equation}\label{g-functions}
  g(x) = \sum_{i=1}^{N_1} w_i \exp \big(- s_i (x-c_i)^2\big)
  + a_0 + a_1 x + a_2 x^2 + a_3 x^3 +
  \sum_{j=1}^{N_2} A_j \sin \big( 2\pi f_j x + \phi_j \big),
\end{equation}
with $x\in \mathbb{R}^d$, $N_1=200$, $N_2=10$, and parameters sampled i.i.d. from simple uniform ranges:
$w_i \sim \mathcal{U}([-1,1]^d)$, $s_i \sim\mathcal{U}([0,50]^d)$, $c_i\sim\mathcal{U}([0,1]^d)$,
$a_0,a_1,a_2,a_3\sim \mathcal{U}([-1,1]^d)$ 
$A_j\sim \mathcal{U}([-0.5,0.5]^d)$, $f_j\sim\mathcal{U}([0.5,8]^d)$ and
$\phi_j\sim \mathcal{U}([0,2\pi]^d)$.
Where necessary, transformations are performed to ensure appropriate scaling of the inputs. This is done in a way that ensures we still capture the richness of the input function space for the operator learning setting. Details are given in each example.
\par In all the examples below, we run multiple instances for the models. For training, we consider a set of pairs $\{g_i, f_i\}_{i=1}^M$, each of which contains data some arbitrary grid $\{x_j\}_{j=1}^N$, resulting in the dataset $\{g_i(x_j), f_i(x_j)\}_{i,j}$, for $i = 1,\dots M$ and $j = 1,\dots N$. To test the models on unseen functions, we generate $M_{test}$ new input functions $g^{test}_i(x)$ from the same family and feed $g^{test}_i$ into the Fredholm NN, $\mathcal{F}(x; g^{test}_i, {K}_{\theta})$ (or the PFNN). We assess the models by comparing the output $\hat{f}_i(x_j)=\mathcal{F}(x_j; g^{test}_i, {K}_{\theta})$, to the corresponding data $f_i(x_j)$, across $i = 1, \cdots, M_{test}$ and $j = 1,\cdots, N$. The relative error metrics as rel. $L_1, L_2$ and $L_{\infty}$ are reported. We present the median and 90/10 percentiles of the errors for all examples.  
\par Furthermore, in all examples we also plot the successive norms $ \hat \rho_{\theta}^{(k)}:= \|\hat{f}^{(k)}_{i} - \hat{f}^{(k-1)}_{i} \|$, where $\hat{f}^{(k)}_{i}(x)$ is the output of layer $k$ of the Fredholm Neural Network, corresponding to the $k-$th iteration of the fixed point estimation. This is used to show that the learned operator is indeed a contraction.
\par All experiments were done in \emph{Python} with an A100 GPU in \emph{Google Colab}. We begin with a linear example, based on an FIE used in \cite{georgiou2025fredholm}.

\subsection{Linear Integral Equations}

To learn the integral operator, we consider some model for the kernel $K_{\theta}(x,z)$, which will be fed into the Fredholm NN, $\mathcal{F}(\cdot; g_i, K_{\theta})$, for a given $g_i$. Learning the operator then consists of adjusting the parameters of the kernel iteratively until the output $\mathcal{F}(x; g_i, K_{\theta})$ is as close as possible to the training data $f_i(x)$, for all pairs of functions in the dataset. For training, the discretized version of the loss function \eqref{loss-linear} is used:
\begin{equation}\label{inverse-loss}
     \hat{ \mathcal{L}}(\theta) = \frac{1}{NM} \sum_{i=1}^{M}\sum_{j=1}^N \Big({f}_i(x_j) - \mathcal{F}(x_j; {K}_\theta, g_i) \Big)^2 + \hat{\mathcal{R}}(\theta),
\end{equation}
where we use the Frobenius norm for the regularization term $\hat{\mathcal{R}}(\theta) = \lambda_{K}\|K_{\theta} \|_F^2$.

\begin{example}\label{ex-linear}
We we solve the inverse  problem, i.e., finding a kernel function $K: [0,1] \times [0, 1] \to {\mathbb R}$ that satisfies the integral equation:
\begin{equation}\label{inverse-prob}
f(x) = g(x) +  \int_0 ^{1} K(x, z) f(z) dz.
\end{equation}

The training dataset consists of $\{g_i(x_j), f_i(x_j)\}_{i,j}$, with $i = 1,\dots M=600$, $j = 1,\dots N=500$, where the $g_i$ functionals are given by \eqref{g-functions} and the solutions $f_i(x)$ are generated using the true kernel: 
\begin{equation}
    K(x,z) = \lambda \big(\cos(25(x-z)) + \cos(7(x-z)) \big).
\end{equation}
Note that, due to linearity, inputs $g$ can be centered and normalized so that $\int_0^1 g(z)dz=0$ and $\|g\|_{L^2([0,1])}=1$. This is done for simplicity in training.
\par As a model of the kernel that will define our learned operator, ${K}_{\theta}(\cdot, \cdot)$, we used a shallow neural network with a single hidden layer, 64 neurons and a $\tanh$ activation function. For the forward pass, we use a Fredholm NN with 20 hidden layers and $N=500$ nodes per layer (corresponding to the $x-$ grid over which we observe the training data), and set $\kappa = 1$ (meaning the Fredholm NN will replicate the Picard iteration). The model is trained using the Adam optimizer and a piecewise decaying learning rate starting from $\lambda = 1$E$-02$, and 50,000 epochs. We perform 30 independent training runs. The mean and median loss is $5.72$E$-07$ and $3.10$E$-07$ respectively. As a first test, we calculate the relative errors across $M_{test} = 20$ test functions and the 30 trained models, on the same grid used for training, which produces a mean rel. $L_1 = 4.89$E$-04$, rel. $L_2 =5.07$E$-04$ and rel. $L_{\infty} = 7.31$E$-04$. Table \ref{tab:errors_for_examples} also shows the median and percentile error metrics across 100 test functions on an unseen grid with $N = 600$. In Fig. \ref{fig:example-1} we display the training loss, a comparison with true solution along with the 90-10 percentile bands, the convergence of towards the estimated solution showing the contractiveness of the learned operator and a contour of the estimated kernel.

\begin{figure}
    \centering
    \includegraphics[width=0.35\textwidth]{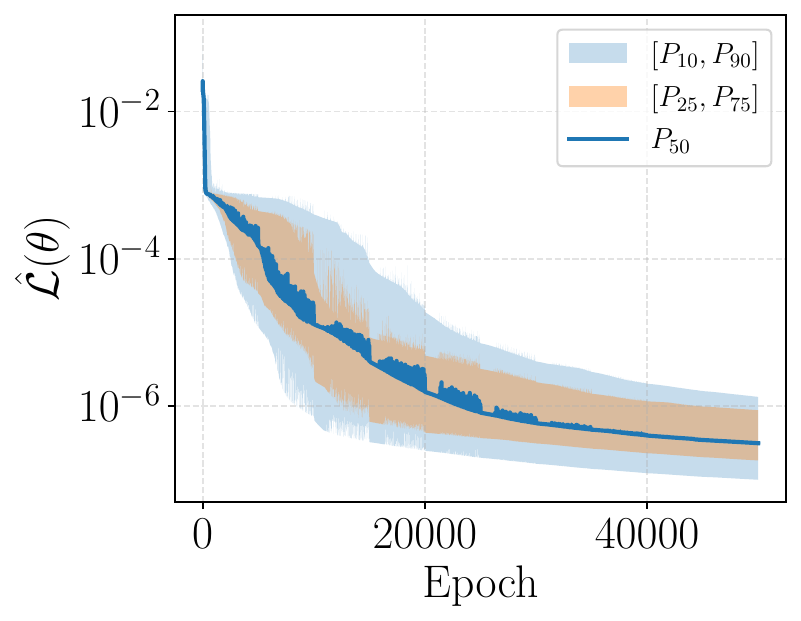}
    \includegraphics[width=0.35\textwidth]{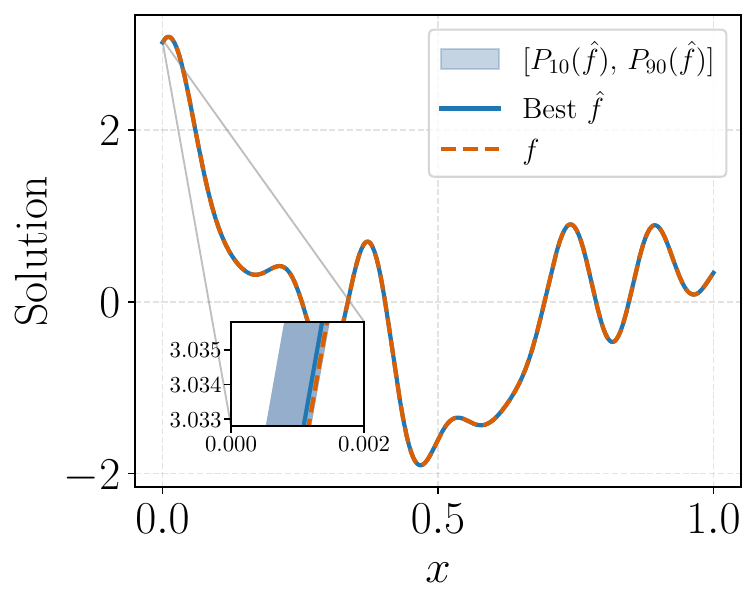}\\
    \includegraphics[width=0.35\textwidth]{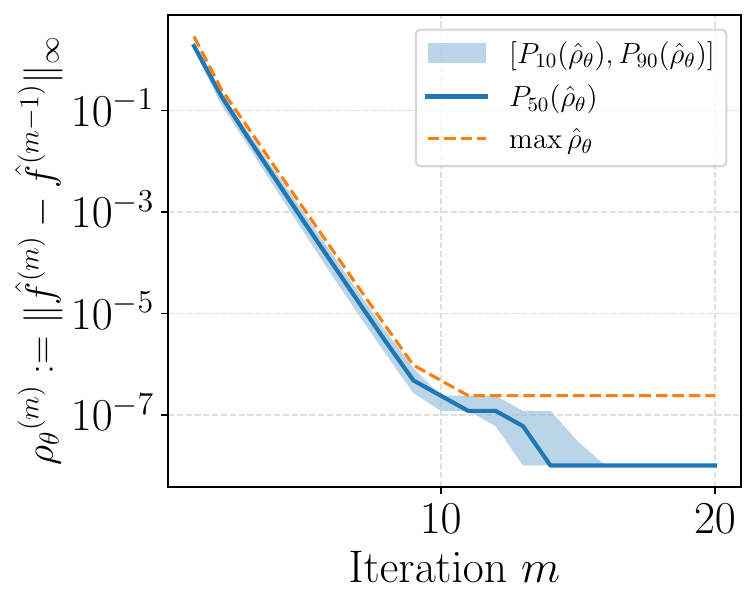}
    \includegraphics[width=0.35\textwidth]{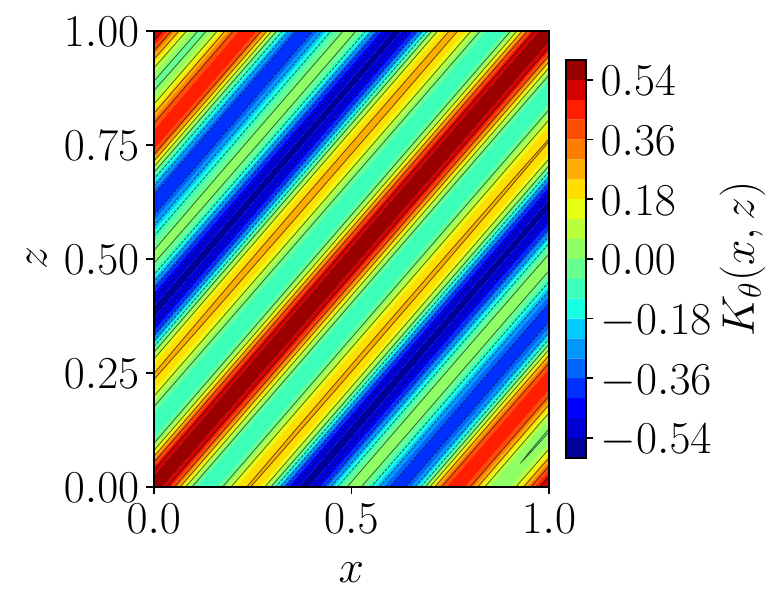}
    \caption{Results for Example \ref{ex-linear}: (Top left) The loss function evolution during training. (Top right) Comparison between the true solution to the integral equation $f_{test}$, for a given $g_{test}$ in the test set, and the estimated solution $\hat{f}_{test}(x)$, as calculated by the forward Fredholm NN with the best (minimum rel. $L_2$ error) learned kernel ${K}_{\theta}$. (Bottom left) The sup-norm of the consecutive layers in the Fredholm NN with the learned kernel, showing that the learned integral operator is a contraction and converges to the fixed point solution. (Right) The contour of the learned kernel model $K_{\theta}.$}
    \vspace{-0.75cm}\label{fig:example-1}
\end{figure}
\end{example}

\subsection{High-dimensional Linear Integral Operators}
We now demonstrate the method in higher dimensions with the example below. 

\begin{example}\label{high-d-example}
Consider the linear FIE \eqref{ie} in $\mathbb{R}^{10}$, and the problem of finding $K : [0,1]^{10} \times [0,1]^{10} \rightarrow \mathbb{R}$. The domain is discretised using $N = 1024$ scrambled Sobol quasi-random nodes and we generate training data $\{g_i,f_i\}_{i=1}^M,$ with $M = 200$, where $g_i$ is of the form \eqref{g-functions} and the corresponding solutions $f_i$ are generated using the forward Fredholm NN with the true kernel:
\begin{equation}\label{eq:kernel-hd}
  K(x, z) = \lambda
  \exp\left(-\sum_{i=1}^{10} \alpha_i\,(x_i - z_i)^2\right)
  \left[1 + \gamma\cos \,\!\bigl(\boldsymbol{\omega} \cdot ({x} - {z})\bigr)\right],
\end{equation}
with $\gamma = 0.5$, anisotropic inverse length scales $\alpha_i$ linearly spaced in the interval $[1.5,3.5]$, and a frequency vector $\boldsymbol{\omega} \in \mathbb{R}^{10}$ drawn as a fixed random unit direction scaled by $6/\sqrt{d}$. The anisotropy in $\alpha_i$ makes the kernel genuinely non-radial. The $1/\sqrt{d}$ scaling of $\boldsymbol{\omega}$ ensures that the argument $\boldsymbol{\omega} \cdot ({x} - {t})$ remains $O(1)$ in all dimensions, and the overall scale $\lambda$ is calibrated so that discrete integral operator $W = K \Delta y$ has spectral norm $\lVert W \rVert_\infty = 0.5$, ensuring the operator contributes meaningfully to the solution while remaining contractive. Note that the training functions $g_i$ are transformed as described in Example \ref{ex-linear}.

The kernel network $K_\theta$ is a $3$-hidden-layer MLP with SiLU activations and widths $[128,128,64]$, receiving as input the concatenated coordinates $({x}, {z}) \in \mathbb{R}^{20}$. For the forward pass, we use a Fredholm NN with 40 hidden layers and $N=1024$ nodes per layer (corresponding to the Sobol samples), and set $\kappa = 1$, as above. The model is trained using the Adam optimizer and a piecewise decaying learning rate starting from $\lambda = 1.0$E$-03$, and 30,000 epochs. Across 30 training instances, the mean and median loss is $1.54$E$-06$ and $1.49$E$-06$ respectively. To test the models we first calculate the relative errors on test functions using the same grid used for training. Across the 30 trained models and the $M_{test} = 100$ test functions, we obtain mean rel. $L_1 = 1.04$E$-03$, rel. $L_2 =1.05$E$-03$ and rel. $L_{\infty} = 1.21$E$-03$. Table \ref{tab:errors_for_examples} shows the median and percentile error metrics across $M_{test} = 100$ test functions on an unseen Sobol grid with $N = 2048$. Comprehensive results are displayed in Fig. \ref{fig:example-high-d}.

\begin{figure}
    \centering
    \includegraphics[width=0.31\textwidth]{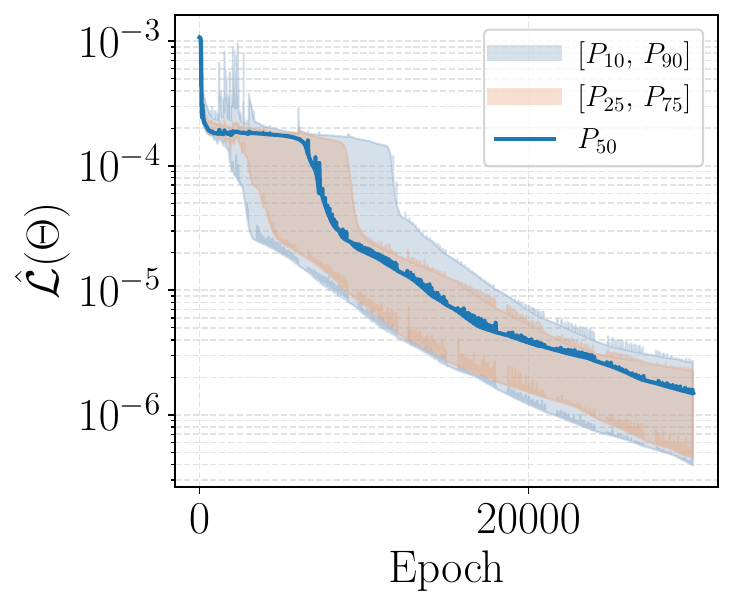}
    \includegraphics[width=0.32\textwidth]{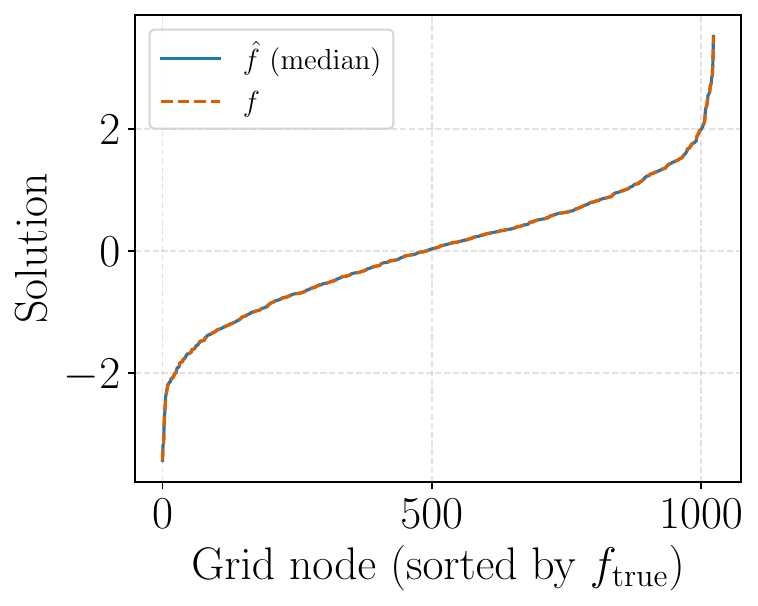}
    \includegraphics[width=0.32\textwidth]{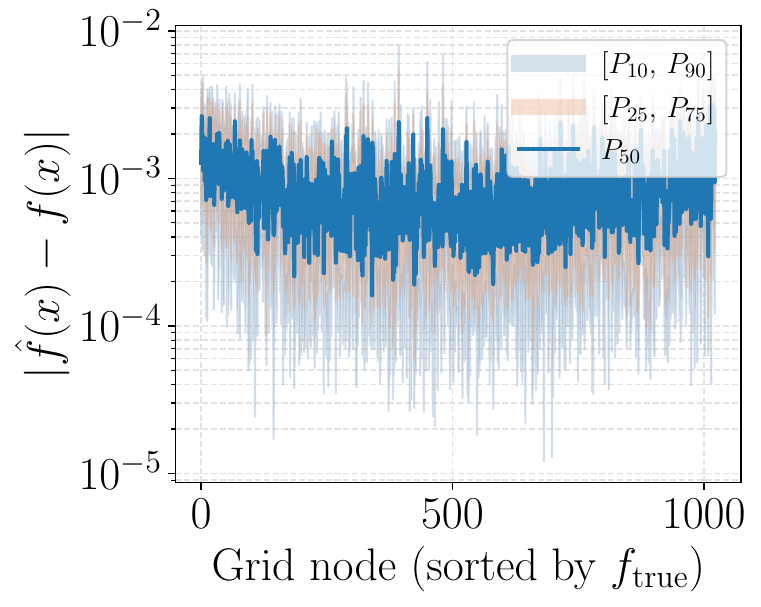}\\
    \includegraphics[width=0.32\textwidth]{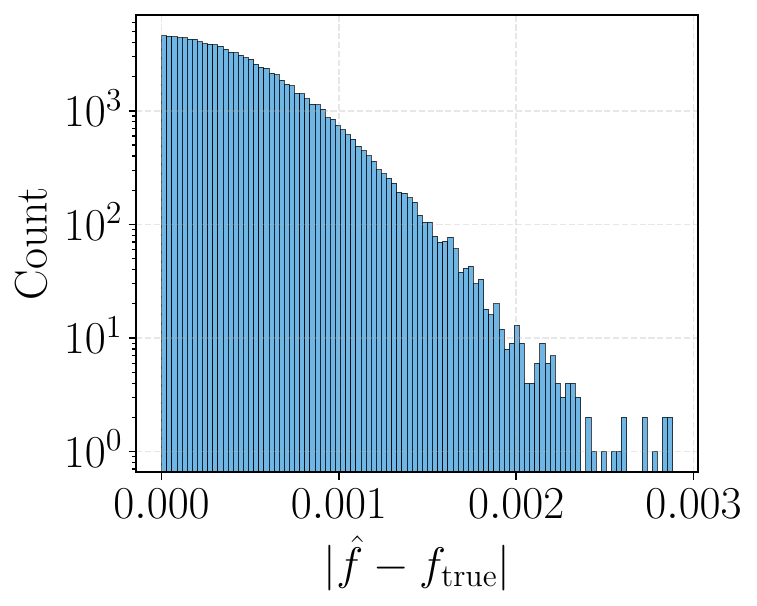}
    \includegraphics[width=0.32\textwidth]{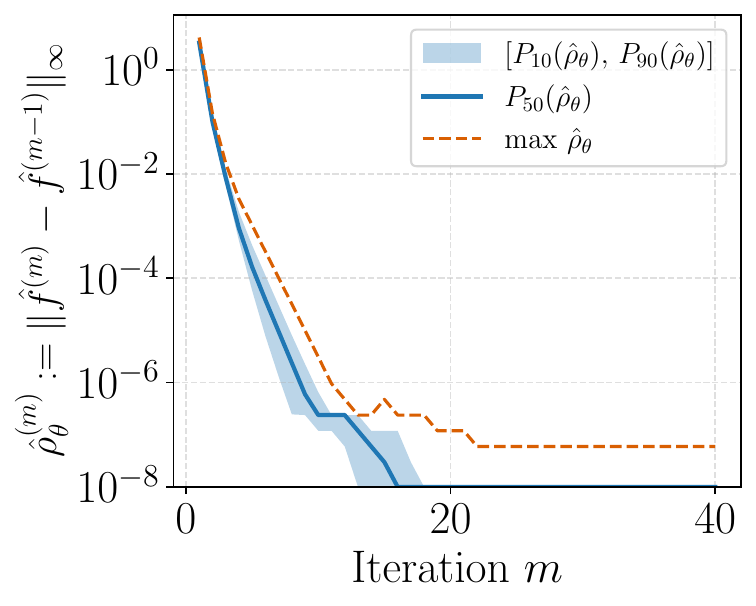}
    \includegraphics[width=0.32\textwidth]{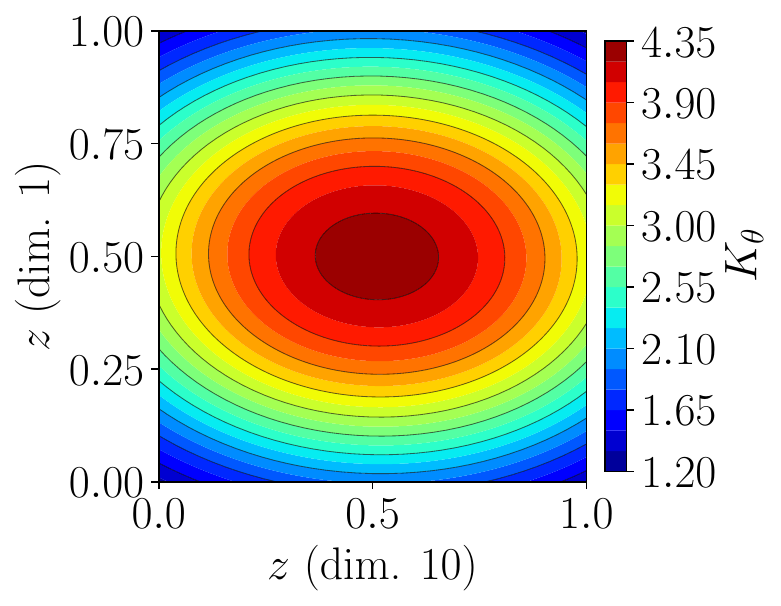}
    \caption{Results for Example \ref{high-d-example}: (Top left) The loss function evolution during training. (Top middle) Comparison between the true solution to the integral equation $f$ in the test set, and the estimated solution $\hat{f}(x)$, as calculated by the forward Fredholm NN with the best (minimum rel. $L_2$ error) learned kernel ${K}_{\theta}$. For the illustration we plot the values at each grid node, with the nodes reordered so that $f(x)$ increases monotonically. (Top right) The plot of the corresponding absolute errors, $|\hat{f}(x) - f(x)|$, at each grid node. (Bottom left) The distribution of the absolute errors for the test functions $g^{test}$, across all $N=1024$ Sobol samples in $\mathbb{R}^{10}$. (Bottom middle) The norm of the consecutive layers in the Fredholm NN with the learned kernel, showing the convergence to the fixed point solution. (Bottom right) The contour of a slice of the learned kernel model $K_{\theta}(x, z) $ with $x = (0.5, \cdots, 0.5)$ and $z = (z_1, 0.5, \dots, 0.5, z_{10})$.}\label{fig:example-high-d}
\end{figure}
\end{example}

\subsection{Non-linear Integral Operators}
For the extension to the non-linear IE of the form \eqref{nl-ie-def} we take advantage of the Recurrent Fredholm NN architecture. 
\par For the numerical implementation, we consider neural network models for the kernel $K_{\theta}(x,y)$, as well as for the non-linearity, ${G}_{\vartheta}: \mathbb{R} \rightarrow \mathbb{R}$. These will be fed into the Reccurent Fredholm NN, $\mathcal{F}(x;g_i, K_{\theta}, {G}_{\vartheta})$, for a given choice of the function $g_i$. The discretized version of the loss function \eqref{loss-non-linear} used for training is given by:
\begin{equation}\label{inverse-loss-nl}
\hat{\mathcal{L}}(\tilde{\theta}) \equiv \hat{\mathcal{L}}(\theta, \vartheta) = \frac{1}{NM} \sum_{i=1}^{M}\sum_{j=1}^N \Big({f}_i(x_j) - \mathcal{F}(x_j;g_i, {K}_\theta, {G}_{\vartheta}) \Big)^2 +\hat{\mathcal{R}}(\theta, \vartheta),
\end{equation}
with $\hat{\mathcal{R}}(\theta, \vartheta) = \lambda_{K}\|K_{\theta} \|^2_F + \lambda_{G}\|G_{\vartheta}\|_F^2$.

\par It is important to note that, in the case of the recurrent Fredholm NN, we have two independent models, $K_{\theta}, G_{\vartheta}$. The standard training approach consists of updating $\theta, \vartheta$ together at each iteration. However, the interaction of the two parameter sets via the integral operator induces a bias, which in turn can lead to poor modelling performance. A better approach is using alternating minimization to optimize the loss function. This approach has been used generally for minimizing non-linear functions (see \cite{bertsekas1999nonlinear}),  and has been also been applied to various machine learning applications such as in \cite{guo2024self, liao2022deep}. According to the Alternating Minimization, each training iteration consists of two phases: one phase where $\theta$ is trained and $\vartheta$ is considered fixed at its current value, and vice versa for the second phase. This is detailed in Algorithm \ref{alg:DTM}. 

\begin{algorithm}[hbt!]\caption{Alternating Minimization for FREDINO training for the non-linear FIE.}\label{alg:DTM}
\begin{algorithmic}
\State {Set total number of epochs.}
\State {Initialize parameters at $\tilde{\theta}^{(0)} = \{\theta^{(0)}, \vartheta^{(0)}\}$ and the corresponding models $K_{\theta^{(0)}}, G_{\vartheta^{(0)}}$.}
\While{$\text{iteration } j \leq \text{total epochs}$}
    \State {\textbf{Step 1. Phase (A)}: Fix $\vartheta^{(j-1)}$ and train to obtain: 
    \begin{equation}\label{inverse-loss-phase-a}
     \theta^{(j)} \approx \operatorname*{argmin}_\theta \hat{\mathcal{L}}(\theta; \vartheta^{(j-1)}),
\end{equation}
where $\mathcal{\hat{L}}(\theta; \vartheta)$ is as given in \eqref{inverse-loss-nl}, but with fixed $\vartheta=\vartheta^{(j-1)}$.
} 
\State {\textbf{Step 2. Phase (B)}: Fix $\theta^{(j)}$ and train to obtain: 
    \begin{equation}\label{inverse-loss-phase-b}
     \vartheta^{(j)} \approx \operatorname*{argmin}_\vartheta \hat{\mathcal{L}}(\vartheta; \theta^{(j)}),
\end{equation}
where $\mathcal{\hat{L}}(\theta; \vartheta)$ is as given in \eqref{inverse-loss-nl}, but with fixed $\theta=\theta^{(j)}$.
} 
\State {\textbf{Step 3.} (Optional) Fine-tune parameters by training jointly $\tilde{\theta} = (\theta, \vartheta)$ (e.g., using L-BFGS or Levenberg - Marquardt.}
    \State {\textbf{Step 4}. Update $j \leftarrow j+1$. }
\EndWhile \\
\Return {Trained parameter set $\tilde{\theta}^{*} =\{\theta^{*}, \vartheta^{*}\}$ and models  $K_{\theta^{*}}, G_{\vartheta^{*}}$.}
\end{algorithmic}
\end{algorithm}

\begin{example}\label{nl-ex}
We consider the problem of learning a non-linear Fredholm Integral operator in $\mathcal{D} = [0,1]$. Given an input $g$, the ground-truth solution $f$ satisfies the nonlinear FIE given by \eqref{nl-ie-def} with
for $x \in [0,1].$ We consider a Gaussian kernel given by:
\begin{equation}
  K(x,z) = C \exp\left(-\frac{(x-z)^2}{2\ell^2}\right),
\end{equation}
with $\ell = 0.2$, and where $C$ is selected so that the spectral norm of the discretized kernel is $0.7$ (again, to ensures that the signal from the kernel is sufficiently strong but remains contractive). The pointwise non-linearity we use is a Gaussian with two peaks, given by:
\begin{equation}\label{example-G}
  G(u) = a_1\exp\left(-\frac{1}{2}\left(\frac{u-m_1}{s_1} \right)\right) + a_2\exp\left(-\frac{1}{2}\left(\frac{u-m_2}{s_2} \right)\right),
\end{equation} 
with $(a_1, a_2) = (0.25, 0.25), (m_1, m_2) = (0.40, -0.40) $ and $(s_1, s_2) = (0.25, 0.15)$.
To produce a diverse set of training inputs sufficiently covering the input domain, whilst also ensuring appropriate scaling for training purposes, the functions $g_i(x)$ are produced as follows: we generate $g'_i$ from \eqref{g-functions}, centered to zero and normalized to unity and we then scale and add an offset, obtaining $g_i = \alpha_i g_i' + \beta_i$, with $\alpha_i, \beta_i \sim \mathcal{U}(-1,1)$. 

For the models $K_{\theta}, G_{\vartheta}$ we use neural networks with 3 and 2 hidden layers, respectively and 64 nodes per layer, both with ReLU activation functions. For the forward pass, a Fredholm NN with $N=500$ nodes per layer and 30 layers was used. For training, we use the alternating minimization scheme as described in Algorithm \ref{alg:DTM}, with regularization constants $\lambda_{K} = \lambda_{G} = 1.0$E$-08$, and a piecewise decaying learning rate from $1.0$E$-03$ to $1.0$E$-06$ with 10,000 total epochs. Across 50 training instances, the mean and median loss is $1.92$E$-08$ and $1.80$E$-08$, respectively. Across a batch of $M_{test} =20$ test functions the models achieve mean rel. $L_1 = 7.79$E$-04$, rel. $L_2 =8.72$E$-04$ and rel. $L_{\infty} = 1.48$E$-03$, when tested on the training grid. As in the previous examples, the median and percentile relative error values when testing on $M_{test} = 100$ test functions on unseen grid of size $N=800$ are given in Table \ref{tab:errors_for_examples}. The graphical results are shown in Fig. \ref{fig:example-3}.


\begin{figure}
    \centering
    \includegraphics[width=0.32\textwidth]{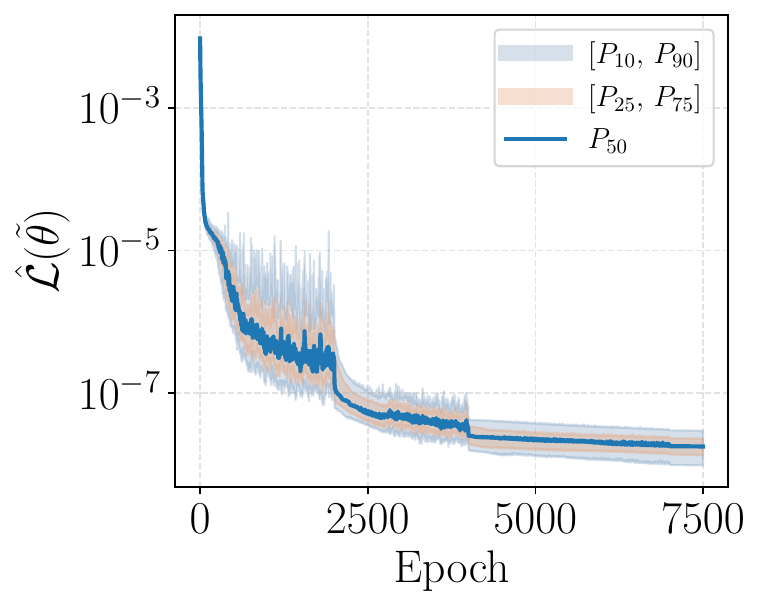}
    \includegraphics[width=0.32\textwidth]{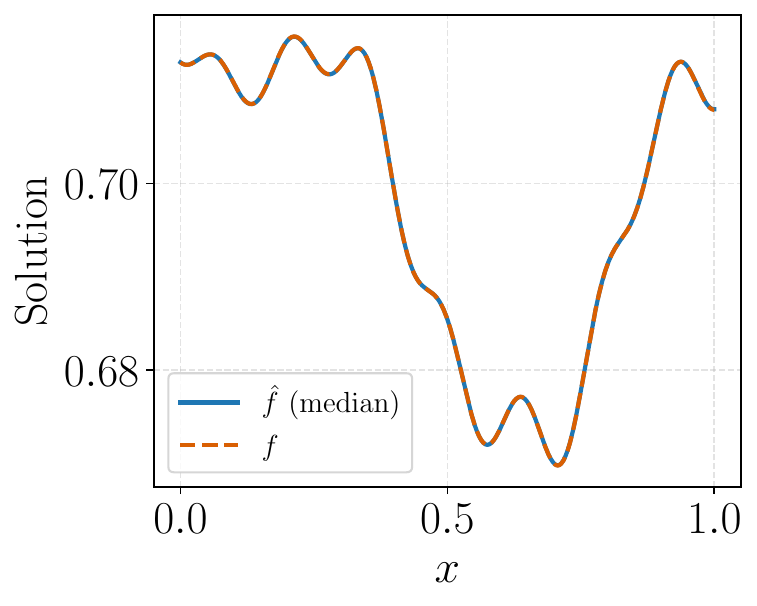}
    \includegraphics[width=0.32\textwidth]{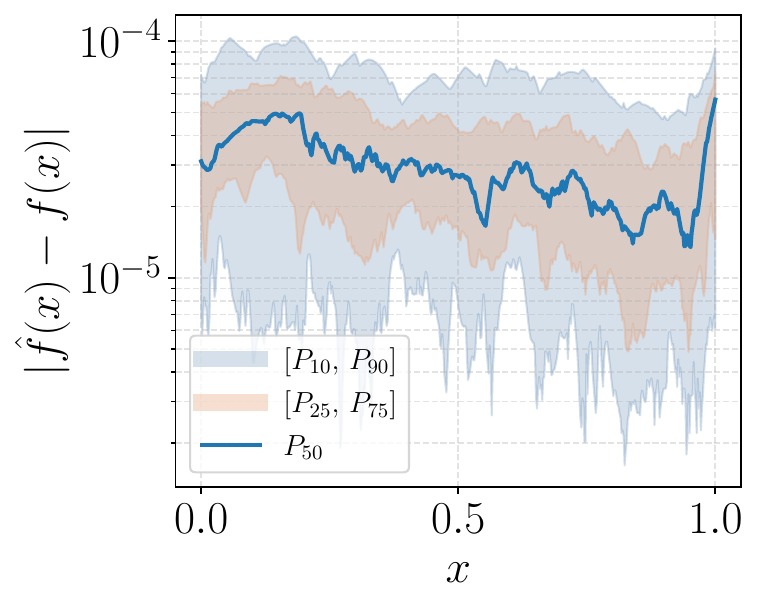}\\
    \includegraphics[width=0.32\textwidth]{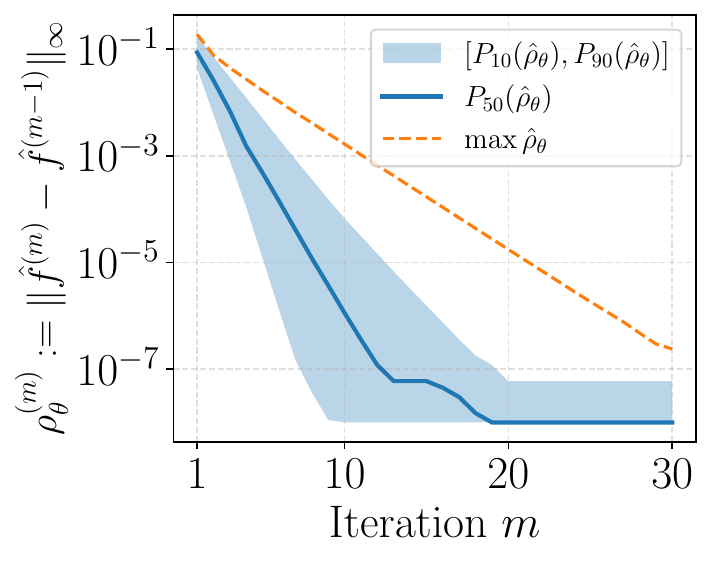}
    \includegraphics[width=0.32\textwidth]{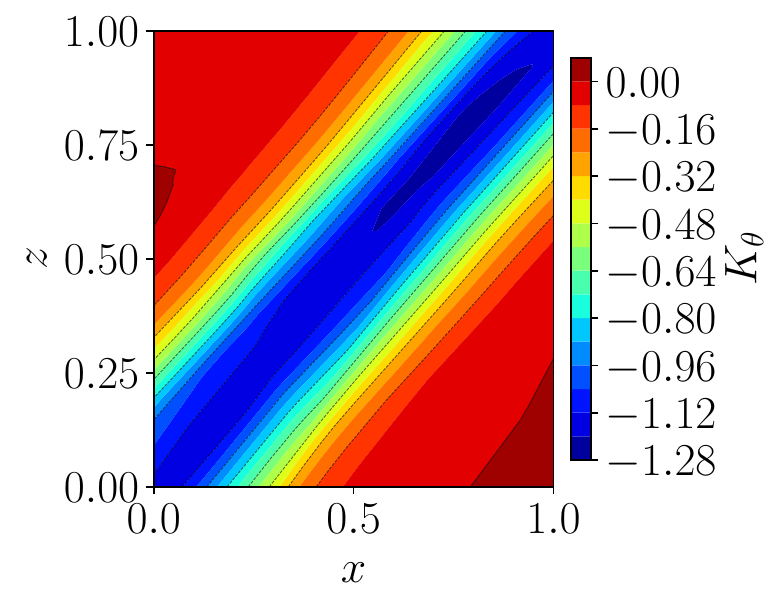}
    \includegraphics[width=0.32\textwidth]{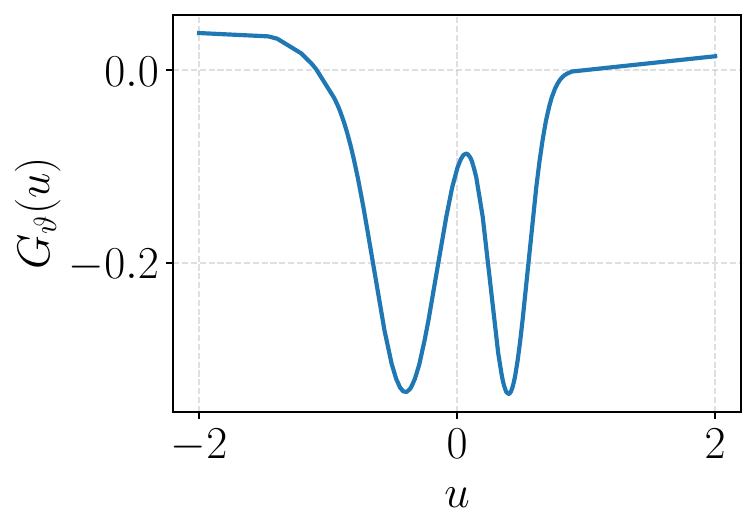}
    \caption{Results for Example \ref{nl-ex}: (Top left) The loss function evolution during training. (Top middle) Comparison between the true solution to the integral equation $f$ for a given $g_i^{test}$ in the test set, and the estimated solution $\hat{f}(x)$, as calculated by the forward Fredholm NN with the best (minimum rel. $L_2$ error) learned kernel $K_{\theta}$. (Top right) The corresponding absolute error plot $|\hat{f}(x) - f(x)|$. (Bottom left) The sup-norm of the consecutive layers in the Fredholm NN with the learned kernel, showing that the learned integral operator is a contraction and converges to the fixed point solution. (Bottom middle - right) The contour of the learned kernel model $K_{\theta}$ and plot of the non-linearity $G_{\vartheta}$ (median values used).}\label{fig:example-3}
\end{figure}
\end{example}


\subsection{High-dimensional Non-Linear Integral Operators}
As in the linear case, we also extend the non-linear FIE example to higher dimensions. We use the same kernel as in the high-dimensional linear FIE example, and the non-linearity is taken as in \eqref{example-G}.

\begin{example}\label{high-d-example-nl}
Consider the linear FIE \eqref{nl-ie-def} in $\mathbb{R}^{10}$, and the problem of finding $K : [0,1]^{10} \times [0,1]^{10} \rightarrow \mathbb{R}$. The modelling parameters are as in example \eqref{high-d-example}, i.e,: we consider a discretised domain with $N = 1024$ scrambled Sobol quasi-random nodes and generate training data $\{g_i,f_i\}_{i=1}^M,$ with $M = 200$, where $g_i$ is of the form \eqref{g-functions} and the corresponding solutions $f_i$ are generated using the forward Fredholm NN with the true kernel as in \eqref{eq:kernel-hd}, with the scale $\lambda$ calibrated so that discrete integral signal is $0.7$. The non-linearity is given as in \eqref{example-G} and the training functions $g_i$ are transformed as described in Example \ref{nl-ex}.

For the kernel network $K_\theta$ we again use a $3$-hidden-layer MLP with SiLU activations and widths $[128,128,64]$, with input $({x}, {z}) \in \mathbb{R}^{20}$. The model for the non-linearity $G_{\vartheta}$ is a neural network with two hidden layers and 64 nodes per layer with a SiLU activation function. For the forward pass, we use a Fredholm NN with 30 hidden layers and $N=1024$ nodes per layer. Training is done with the two-phase approach described in \ref{alg:DTM} using the Adam optimizer and a piecewise decaying learning rate starting from $\lambda = 1.0$E$-03$, and 10,000 epochs, followed by an L-BFGS fine tuning with 500 epochs. Across 30 training instances, the mean and median loss is $6.58$E$-07$ and $4.82$E$-07$, respectively. The relative errors calculated on the same grid used for training, and aggregated over the 30 model instances and $M_{test} = 100$ test functions result in mean rel. $L_1 = 1.41$E$-03$, rel. $L_2 =1.82$E$-03$ and rel. $L_{\infty} = 8.36$E$-03$. As in the previous examples, in Table \ref{tab:errors_for_examples} we also provide the median and percentile error metrics across $M_{test} = 100$ test functions on an unseen Sobol grid with $N = 2048$. The results are also displayed in Fig. \ref{fig:example-high-d-nl}.

\begin{figure}
    \centering
    \includegraphics[width=0.24\textwidth]{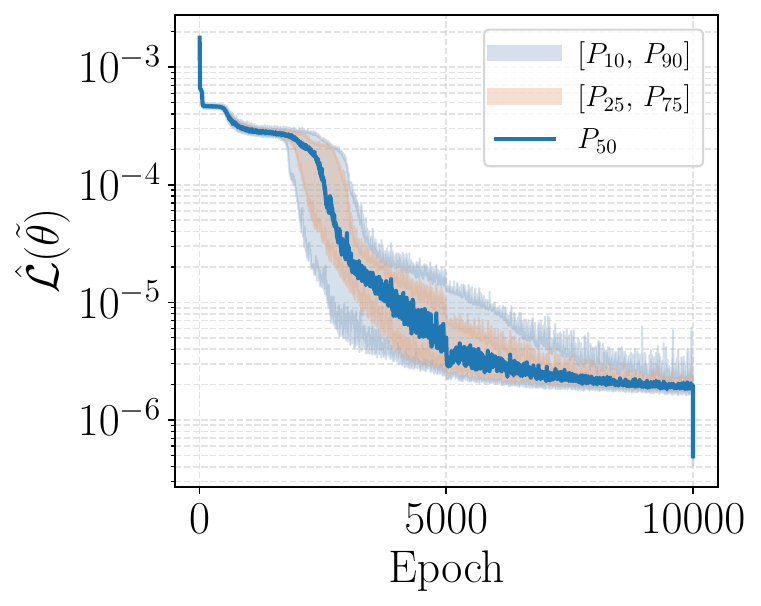}
    \includegraphics[width=0.24\textwidth]{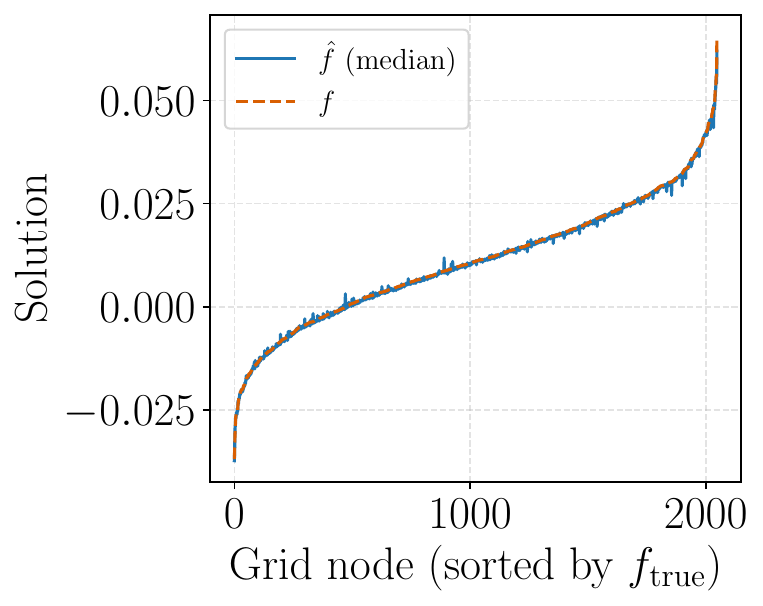}
    \includegraphics[width=0.24\textwidth]{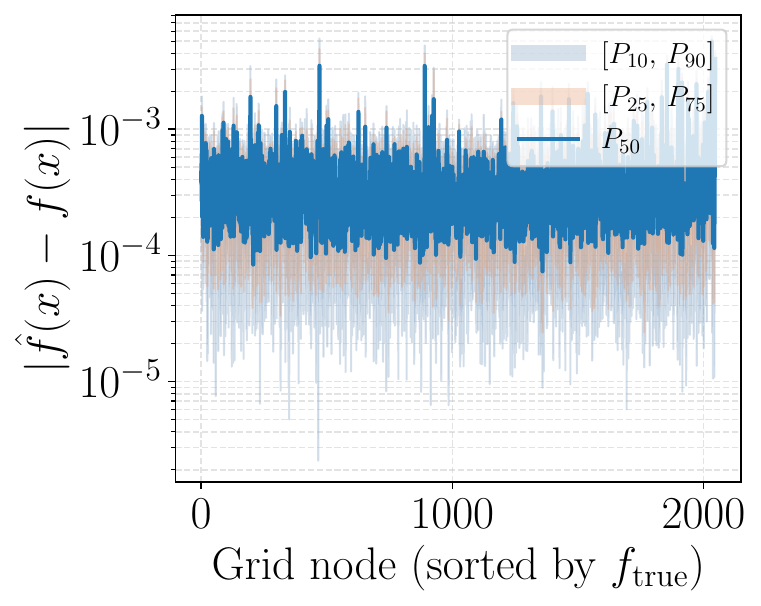}
    \includegraphics[width=0.24\textwidth]{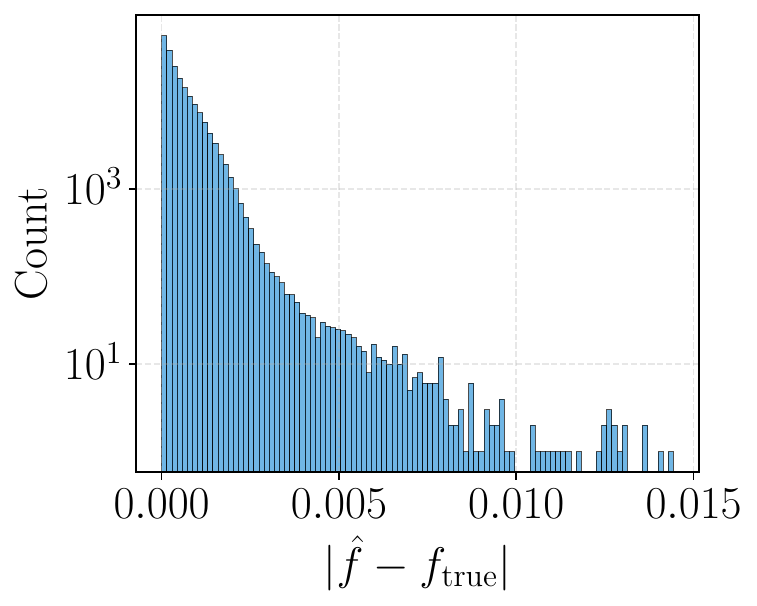}\\
    \includegraphics[width=0.3\textwidth]{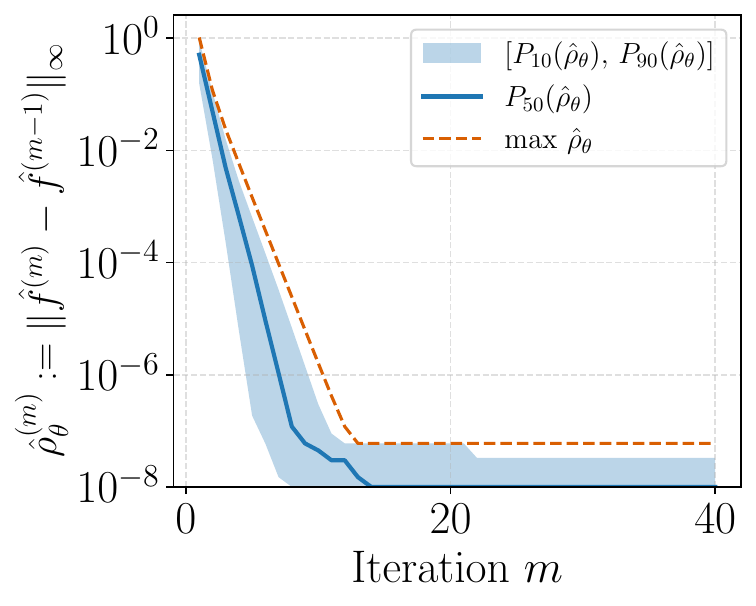}
    \includegraphics[width=0.3\textwidth]{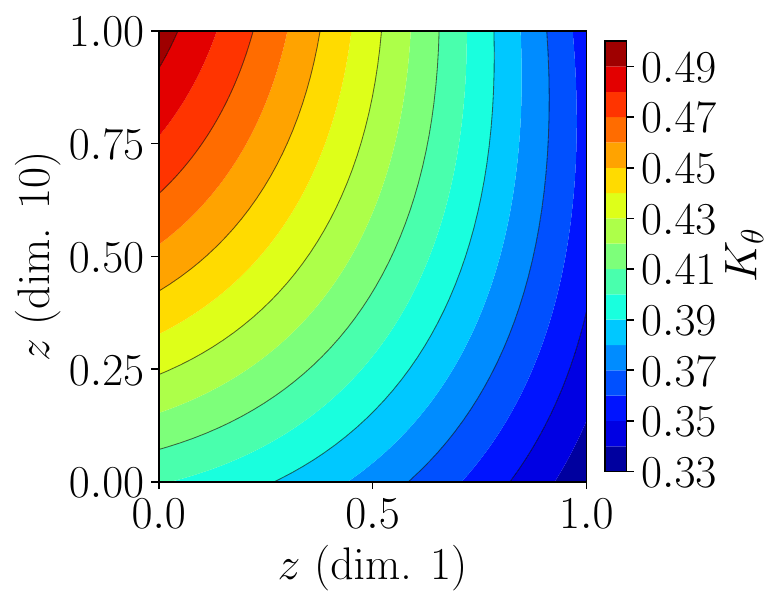}
    \includegraphics[width=0.3\textwidth]{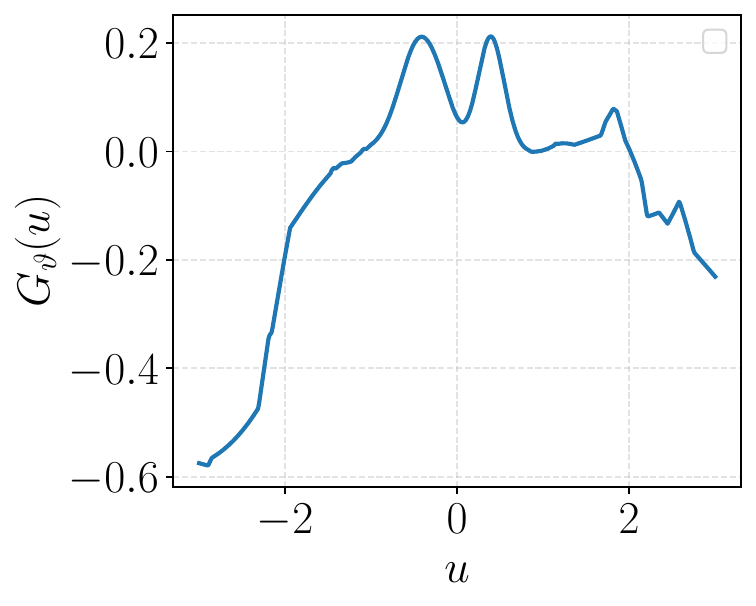}
    \caption{Results for Example \ref{high-d-example-nl}: (Top row, from left to right) a) The loss function evolution during training. b) Comparison between the true solution to the integral equation $f$ in the test set, and the estimated solution $\hat{f}(x)$, as calculated by the forward Fredholm NN with the best (minimum rel. $L_2$ error) learned kernel ${K}_{\theta}$ and non-linearity $G_{\vartheta}$. Again, for illustration we plot the values at each grid node, with the nodes reordered so that $f(x)$ increases monotonically. Here, we use the unseen test grid. c) The plot of the corresponding absolute errors, $|\hat{f}(x) - f(x)|$. d) The distribution of the absolute errors for the test functions $g^{test}$, across all $N=1024$ Sobol samples in $\mathbb{R}^{10}$. (Bottom left) The norm of the consecutive layers in the Fredholm NN with the learned kernel, showing the convergence to the fixed point solution. (Bottom middle) The contour of a slice of the learned kernel model $K_{\theta}(x, z) $ with $x = (0.5, \cdots, 0.5)$ and $z = (z_1, 0.5, \dots, 0.5, z_{10})$. (Bottom left) The plot of the learned non-linearity $G_{\vartheta}$ (median values used).}\label{fig:example-high-d-nl}
\end{figure}
\end{example}

\begin{table}[h]
\centering
\begin{tabular}{l|cc|cc|cc}
\toprule
 & \multicolumn{2}{c}{Rel $L_1$} & \multicolumn{2}{c}{Rel $L_2$} & \multicolumn{2}{c}{Rel $L_\infty$} \\
\cmidrule(lr){2-3}\cmidrule(lr){4-5}\cmidrule(lr){6-7}
\textbf{Example} & \textbf{Median} & \textbf{$[P_{10}, P_{90}]$} 
                 & \textbf{Median} & \textbf{$[P_{10}, P_{90}]$} 
                 & \textbf{Median} & \textbf{$[P_{10}, P_{90}]$} \\
\midrule
\ref{ex-linear} 
  & 4.15E-04 & [1.84E-04, 1.00E-03]
  & 4.38E-04 & [1.97E-04, 1.05E-03]
  & 6.31E-04 & [3.07E-04, 1.48E-03] \\
\ref{high-d-example} 
  & 7.31E-04 & [4.44E-04, 1.12E-03]
  & 7.31E-04 & [4.53E-04, 1.17E-03]
  & 1.06E-04 & [5.47E-04, 2.41E-03] \\
\ref{nl-ex} 
  & 2.29E-04 & [3.91E-05, 1.12E-03]
  & 2.75E-04 & [4.57E-05, 1.26E-03]
  & 6.11E-04 & [8.35E-05, 1.87E-03] \\
\ref{high-d-example-nl}
  & 1.82E-03 & [2.17E-03, 4.68E-03]
  & 2.63E-03 & [2.87E-03, 6.20E-03]
  & 2.12E-02 & [1.60E-02, 4.18E-02] \\
  
\bottomrule
\end{tabular}
\caption{For each model istance, $j = 1,\dots, R$, the relative error metrics are computed per test function, $g_i, i = 1,\dots, M_{test}$ on the unseen grid of size $N$. The statistics are reported across all $R \times M_{test}$ model and test function combinations.}
\label{tab:errors_for_examples}
\end{table}

\subsection{Learning semi-linear elliptic PDE solution operator}
Finally, we implement the approach detailed in section \ref{pde-sec}. Here, the dataset consists of the set of boundary condition and solution functions $\{g_i(x^*_j), u_i(x_j)\}$, $i= 1, \dots, M$, $j=1,\dots N$, and where $x_j = (r_j \cos(\theta_j), r_j \sin(\theta_j)) \in \Omega$ and $x^*_j = (\cos(\theta_j), \sin(\theta_j))\in \partial \Omega$. The model $\Phi_{\theta}$ is trained by minimizing the discretized version of \eqref{loss-pfnn} (equivalently \eqref{loss-pfnn-2}):
\begin{equation}
     \hat{\mathcal{L}}(\theta) = \frac{1}{MN} \sum_{i = 1}^M \sum_{j = 1}^N \big(u_i(x_j) - \mathcal{P}^{\mathcal{F}}(x_j;g,\Phi_{\theta})\big)^2 + \hat{\mathcal{R}}(\theta).
\end{equation}
We implement this approach using a modified example of that studied in \cite{georgiou2025fredholm2}.

\begin{example}\label{ex-nl-pde}
Consider now the 2D PDE:
\begin{eqnarray}\label{nl-pde}
\begin{cases}
 \mathcal{A} u(x)  = \psi(u(x),x), \quad x \in \Omega= \{(x_1,x_2 )\in \mathbb{R}^2 : x_1^2 + x_2^2 \leq 1 \} \\ 
u(x) = g(x), \quad x \in \partial \Omega,
\end{cases}
\end{eqnarray}
with $\psi(u, x)= \tanh(u) - e^{1-x_1^2-x_2^2} + 4$, and the problem of learning the elliptic operator $\mathcal{A}$. For the basis of our model we consider the (elliptic) Helmholtz operator $(\Delta - \lambda)$ and learn an appropriate correction term to the fundamental solution. Hence, we will use the representations as given in \eqref{bie1} and \eqref{bie2}, with $\tilde{\psi}(u,x) = -\lambda u + \psi(u,x)$.
\par To ensure consistency with the ``physics'' of the potential integral and the jump condition, we consider a model that exhibits the weak singularity as we approach the boundary. This is achieved by using the model that depends on the fundamental solution of the Helmholtz operator (which depends on $K_0(r)$, with $r=|x-y|$, which exhibits a singularity as $r\rightarrow 0$). In order to avoid numerical issues due to the singularity, which can affect the minimization algorithm during model training, we use a smoothed approximation of the function, by taking into account its asymptotics (\cite{abramowitz1948handbook}). In particular, we use: 
\begin{equation}
    \tilde{K}_0(r) \approx \big(1-\sigma(r)\big)\left(-\ln(\frac{r}{2}) - \gamma \right) + \sigma(r)\left(\sqrt{\frac{\pi}{2r}}e^{-r} \right),
\end{equation}
where $\gamma$ is the Euler-Mascheroni constant and $\sigma(r) = \frac{1}{1+ e^{-5(r-1.5)}}$ is the transformed sigmoid function (of course, many different configurations are possible). This expression follows from the fact that $K_0(r) \approx -\ln(\frac{r}{2})-\gamma + O(r^2 \ln r)$, as $r\rightarrow 0$ and 
$K_0(r) \approx \sqrt{\frac{\pi}{2r}}e^{-r}(1 + O(1/r))$, as $r\rightarrow \infty$. 
\par We therefore use $\Phi(x,y) \approx -\frac{1}{2\pi}\tilde{K}_0(\sqrt{\lambda}|x-y|)$, and consider a shallow neural network for the correction term, so that the final model is a linear combination of the form:
\begin{equation}
    \Phi_{\theta}(x,y) = \Phi(x,y) + \alpha C_{\theta}(x,y),
\end{equation}
where $\alpha \in (0,1)$ is the mixing constant (note that we let $\alpha$ be a learnable parameter to enhance model flexibility) and $C_{\theta}$ is the trainable model that captures the correction for the given nonlinear PDE. The derivative $\frac{\partial \Phi_{\theta}}{\partial n_y}$ can then be calculated using automatic differentiation. 
\par For our experiments, $C_{\theta}$ is a neural network with 2 hidden layers, 64 neurons per layer and a $\tanh$ activation function. For training, we use $M = 200$ realizations of the boundary functions $g_i(x^*)$. These functions are generated from \eqref{g-functions}, transformed to have zero mean and such that $\| g_i\|_{\infty} = \alpha_i$, with $\log(\alpha_i) \sim \mathcal{U}(\log(0.05), \log(2.0))$. The corresponding solutions for training, $u_i(x)$, are generated using the iterative process \eqref{nl-pde2} as described in \cite{georgiou2025fredholm2}. The model is trained for 200 epochs, using a piecewise decaying learning rate starting at $1.0$E$-03$ and reaching $1.0$E$-06$. Finally, we use a light regularization constant $\lambda_K = 1.0$E$-10$. Across 20 training instances, the mean and median training loss is $3.35$E$-06$ and $2.78$E$-06$. We use $M_{test }= 20$ test boundary functions $g_i^{test}$ and perform two tests: firstly, we use the learned $\Phi_{\theta}$ along with the true solution $u_i^{test}$ to calculate the solution $\hat{u}_i^{test}$ via \eqref{bie1} and compare to the true data. We refer to this as the reconstruction test.
Secondly, we use the learned $\Phi_{\theta}$ to construct the PFNN $\mathcal{P}^{\mathcal{F}}(\cdot;g^{test},\Phi_{\theta})$, and perform the full Picard iteration scheme to solve the PDE and compare to the true solutions (i.e., true solver test). The error metrics for the two tests are given in Table \ref{tab:helmholtz_errors}.
In Fig. \ref{fig:example-4} we display the loss function, the fixed point convergence of the learned boundary function $\beta$ and the the learned kernel $\Phi_{\theta}$. These plots show that we have obtained the required properties, i.e., contractiveness and the numerical blow-up of the fundamental solution on the diagonal. Finally, we also present an example of the learned solution contour and the corresponding absolute error, showing that the larger absolute errors appear near regions where the solution approaches zero, as expected.
\begin{figure}
    \centering
    \includegraphics[width=0.31\textwidth]{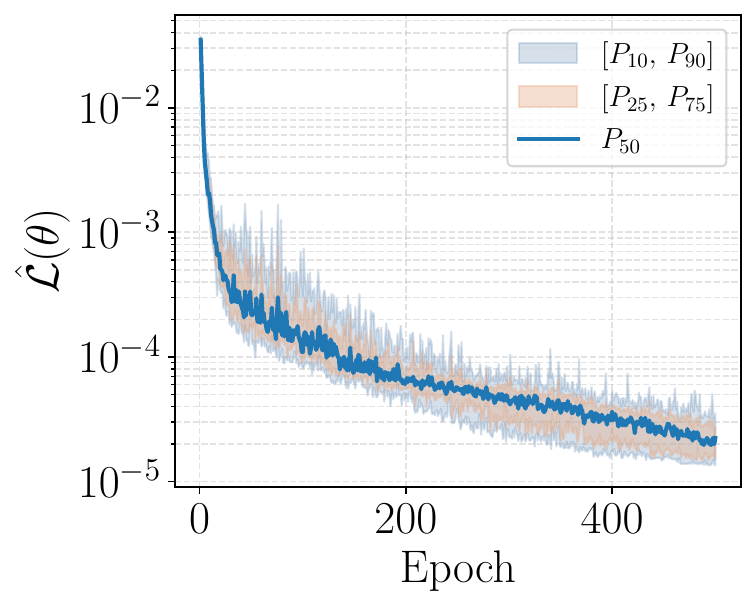}
    \includegraphics[width=0.32\textwidth]{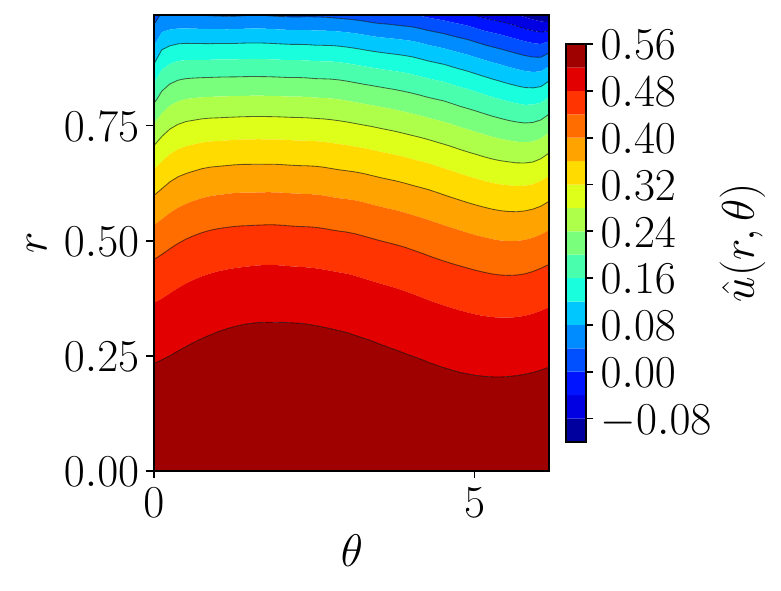} 
    \includegraphics[width=0.31\textwidth]{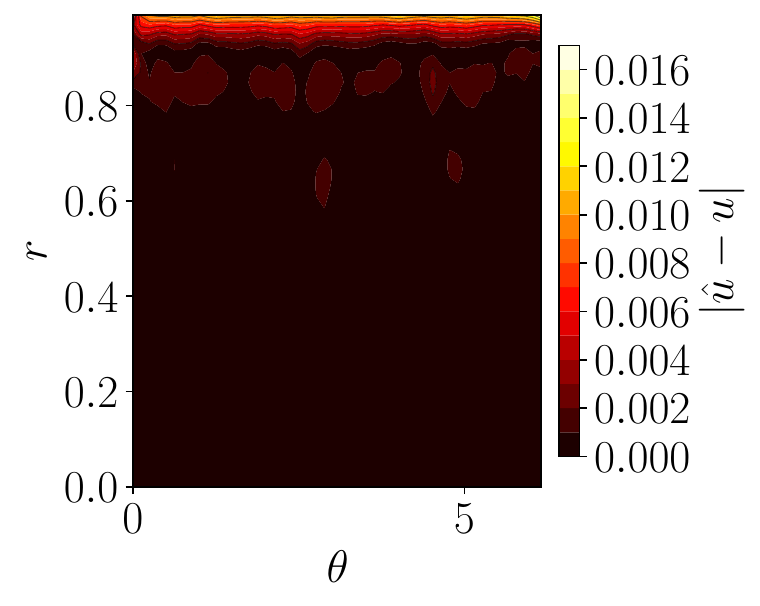}
    \includegraphics[width=0.32\textwidth]{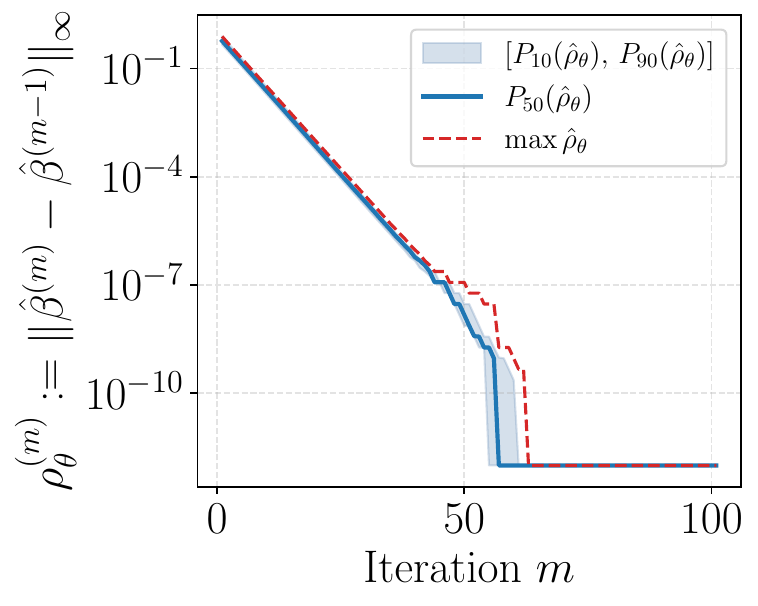}
    \includegraphics[width=0.32\textwidth]{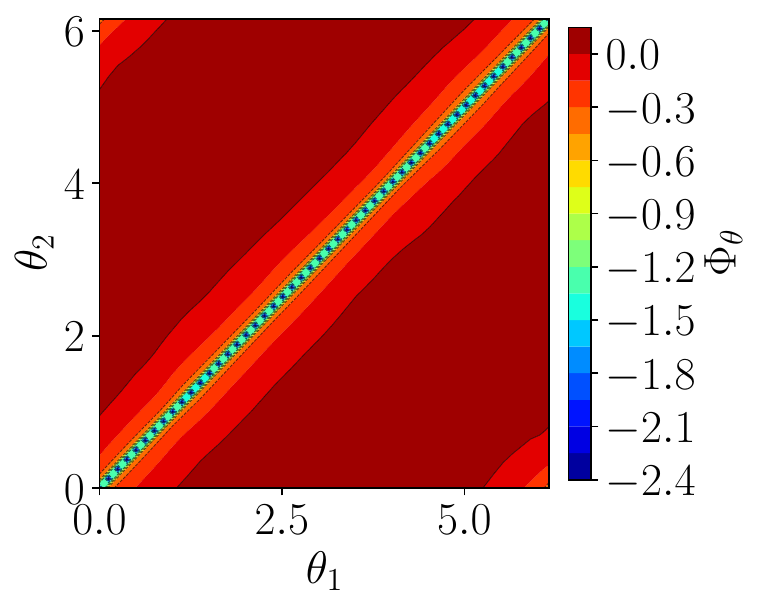}
    \caption{Results for Example \ref{ex-nl-pde}: (Top left) The loss function evolution during training. (Top middle) Estimated $\hat{u}(r,\theta)$, as calculated by the forward Potential Fredholm NN with the best (minimum rel. $L_2$ error) learned model ${\Phi}_{\theta}$, for a given test boundary function $g^{test}$. (Top right) Absolute error between the approximated solution $\hat{u}$ and the true test data $u$, for $g^{test}$. (Bottom left) The sup-norm of the consecutive layers in the Fredholm NN with the learned kernel, showing that the learned integral operator is a contraction and converges to the fixed point solution. (Right) The contour of the learned kernel model $K_{\theta}.$} \label{fig:example-4}
\end{figure}
\end{example}
\begin{table}[h]
\centering
\begin{tabular}{ll|cc|cc}
\toprule
 & & \multicolumn{2}{c|}{Rel $L_1$} & \multicolumn{2}{c}{Rel $L_2$} \\
\cmidrule(lr){3-4}\cmidrule(lr){5-6}
& \textbf{Test} & \textbf{Median} & \textbf{$[P_{10}, P_{90}]$} 
& \textbf{Median} & \textbf{$[P_{10}, P_{90}]$} 
\\
\midrule
\multirow{2}{*}{\rotatebox[origin=c]{90}{\scriptsize Run}}
& Recon. 
  & 4.39E-03 & [3.22E-03, 5.28E-03]
  & 7.72E-03 & [6.70E-03, 9.34E-03]
   \\
& Solver
  & 4.78E-03 & [3.48E-03, 5.91E-03]
  & 7.86E-03 & [6.74E-03, 9.51E-03]
  \\
\midrule
\multirow{2}{*}{\rotatebox[origin=c]{90}{\scriptsize Sample}}
& Recon. 
  & 3.84E-03 & [2.47E-03, 7.18E-03]
  & 6.74E-03 & [5.22E-03, 1.22E-02]
   \\
& Solver
  & 4.20E-03 & [2.60E-03, 7.92E-03]
  & 7.05E-03 & [5.24E-03, 1.25E-02]
  \\
\bottomrule
\end{tabular}
\caption{Errors for Example \ref{ex-nl-pde}: per-run metrics are calculated over the 20 averages from the 20 test functions. Per-sample statistics are calculated over all $20 \times 20 = 400$ evaluations. }
\label{tab:helmholtz_errors}
\end{table}

\section{Conclusions \& Discussion}\label{sec6}
In this work, building on the Fredholm  NN framework that we introduced recently in \cite{georgiou2025fredholm, georgiou2025fredholm2}, we introduce the Fredholm Integral Neural OperatorS (FREDINOs) to learn linear and non-linear contractive integral operators in arbitrary dimensions. We also show how the framework can be exploited in the setting of linear and non-linear elliptic PDEs. The satisfaction of the contraction property is achieved by the forward pass through the Fredholm NN that is performed during the training step. Hence, in this way, we are able to obtain the model that approximates the integral operator and simultaneously obtain the full structure of the trainable Fredholm NN that is used in the forward pass. This is equivalent to learning the full fixed-point algorithm. The approach highlights how incorporating explaibable numerical analysis-informed deep neural network architectures can allow us to obtain functions and operators that adhere to specific mathematical properties. We demonstrate the performance of FREDINOs via linear and non-linear integral operators of arbitrary dimensions, as well as in the context of a semi-linear elliptic PDE in 2D.
\par From the proposed methods, interesting future research directions also arise, with one of the most interesting being the application of FREDINOs to higher-dimensional PDEs. In particular, for $d\geq 3$ the Potential Fredholm Neural Network can still be constructed using the existing theory (see \cite{fabes1978potential, verchota1984layer}). However, the numerical scheme requires special treatment, as the fundamental solutions change from logarithmic functions of the form to $\Phi(x,y) \propto |x-y|^{-(d-2)}$. Hence, in addition to the known curse of dimensionality problem, we also have to deal with the dimension-dependent singular kernels. This will be the focus of future work.  

\section*{Acknowledgments}
K.C. G. acknowledges support from the PNRR MUR Italy, project PE0000013-Future Artificial Intelligence Research-FAIR. 
C.S. acknowledges partial support from the PNRR MUR Italy, projects PE0000013-Future Artificial Intelligence Research-FAIR \& CN0000013 CN HPC - National Centre for HPC, Big Data and Quantum Computing, and from the GNCS group of INdAM. A.N.Y. acknowledges the use of resources from the Stochastic Modelling and 
Applications Laboratory, AUEB.  
\bibliographystyle{plain}
\bibliography{Fredholm_Neural_Operators}

\end{document}